\documentclass[11pt]{article}
\usepackage{latexsym,amsfonts,amssymb,amsmath,amsthm}
\usepackage{graphicx}
\usepackage{subcaption}
\usepackage[usenames,dvipsnames]{color}
\usepackage{ulem}
\usepackage{caption}
\usepackage{float}
\usepackage{xcolor}
\usepackage{bbm}
\usepackage{comment}
\usepackage{appendix}

\begin{document}

\newtheorem{tm}{Theorem}[section]
\newtheorem{prop}[tm]{Proposition}
\newtheorem{defin}{Definition}[section]
\newtheorem{coro}{Corollary}[section]
\newtheorem{lem}{Lemma}[section]
\newtheorem{assumption}{Assumption}[section]
\newtheorem{rk}{Remark}[section]
\newtheorem{nota}{Notation}[section]
\numberwithin{equation}{section}

\newcommand{\mg}{\color{RedViolet}}

    \newcommand{\lb}{\label}
  \newcommand{\beq}{\begin{equation}}
    \newcommand{\eeq}{\end{equation}}

\newcommand{\stk}[2]{\stackrel{#1}{#2}}
\newcommand{\dwn}[1]{{\scriptstyle #1}\downarrow}
\newcommand{\upa}[1]{{\scriptstyle #1}\uparrow}
\newcommand{\nea}[1]{{\scriptstyle #1}\nearrow}
\newcommand{\sea}[1]{\searrow {\scriptstyle #1}}
\newcommand{\csti}[3]{(#1+1) (#2)^{1/ (#1+1)} (#1)^{- #1
 / (#1+1)} (#3)^{ #1 / (#1 +1)}}
\newcommand{\RR}[1]{\mathbb{#1}}

\newcommand{\rd}{{\mathbb R^d}}
\newcommand{\ep}{\varepsilon}
\newcommand{\rr}{{\mathbb R}}
\newcommand{\alert}[1]{\fbox{#1}}
\newcommand{\eqd}{\sim}

\def\un{\underline}
\def\ba{\overline}

\def\p{\partial}
\def\R{{\mathbb R}}
\def\N{{\mathbb N}}
\def\Q{{\mathbb Q}}
\def\C{{\mathbb C}}
\def\l{\left(}
\def\r{ \right) }
\def\t{\tau}
\def\k{\kappa}
\def\a{\alpha}
\def\la{\lambda}
\def\De{\Delta}
\def\de{\delta}
\def\ga{\gamma}
\def\Ga{\Gamma}
\def\ep{\varepsilon}
\def\eps{\varepsilon}
\def\si{\sigma}
\def\Re {{\rm Re}\,}
\def\Im {{\rm Im}\,}
\def\E{{\mathbb E}}
\def\P{{\mathbb P}}
\def\Z{{\mathbb Z}}
\def\D{{\mathbb D}}
\def\calS{\mathcal{S}}
\def\O{{\Omega_T}}
\def\calC{\mathcal{C}}
\def\calL{\mathcal{L}}
\def\calX{\mathcal{X}}

\newcommand{\ceil}[1]{\lceil{#1}\rceil}

\newcommand{\bbR}{{\mathbb R}}

\newcommand{\calB}{{\mathcal B}}

\title{Chemotaxis Models with Nonlinear/Porous Medium Diffusion,
Consumption, and Logistic source on 
$\mathbb{R}^N$:  I.  Global Solvability and Boundedness}

\author{
Zulaihat Hassan,  Wenxian Shen, and Yuming Paul Zhang  \\
Department of Mathematics and Statistics\\
Auburn University, AL 36849\\
U.S.A. }

\date{}
\maketitle

\begin{abstract}
This series of  papers is concerned with  the global solvability,  boundedness, regularity, and uniqueness of weak solutions to the following parabolic-parabolic chemotaxis system with a logistic source and chemical consumption:
\begin{equation*}
\begin{cases}
u_t = m\nabla\cdot \left((\eps+u)^{m-1}\nabla u\right) - \chi \nabla \cdot (u \nabla v) + u(a - b u), & \text{ in } (0,\infty)\times\mathbb{R}^N, \\
v_t = \Delta v - uv, & \text{ in } (0,\infty)\times\mathbb{R}^N,
\end{cases}
\end{equation*}
where $m > 1$ and $\eps \geq 0$.
The present paper focuses on the global solvability and boundedness of weak solutions.  For general bounded initial data, which may be non-integrable, we prove the existence of global weak solutions that remain uniformly bounded for all times. The proof relies on deriving local $L^p$ estimates that are uniform in time via a new continuity-type argument and obtaining $L^\infty$ bounds using Moser's iteration; all of these estimates are uniform as $\eps\to0$.  
{In part II,} we will study the regularity and uniqueness of weak solutions.

\end{abstract}

\medskip


\medskip


\tableofcontents
\section{Introduction}

\subsection{Overview}

The movement of biological species is often influenced by specific chemicals, which play a crucial role in various biological processes such as bacterial aggregation, immune responses, and angiogenesis during embryonic development and tumor progression. The general chemotaxis system that describes the movement of cells in response to a chemical signal is given below,
\begin{equation}
\label{general-eq}
\begin{cases}
\p_t u=\nabla \cdot\big(D(u,v)\nabla u-\chi(u,v)\nabla v\big)+f(u,v),\quad & x\in\Omega\cr
\tau \p_t v=d \Delta v+g(u,v)-h(u,v)v,\quad &x\in\Omega.
\end{cases}
\end{equation}
The variable \( u \) represents the cell density, while \( v \) denotes the concentration of the chemical signal in a given domain \( \Omega \subset \mathbb{R}^N \) with $N\geq 1$, which can be bounded or unbounded. The function \( D(u,v) \) describes the diffusivity of the cells, and \( \chi(u,v) \) represents the chemotaxis sensitivity. The function \( f(u,v) \) models cell growth and death, whereas the functions \( g(u, v) \) and \( h(u, v) \) describe the production and degradation of the chemical signal, respectively.  The parameters \( \tau \) and \( d \) are associated with the diffusion speed of the chemical substance.

In the current series of papers,  we consider a chemotaxis system in which the diffusivity of the cells is governed by a nonlinear function of the cell density. Specifically, we study system \eqref{general-eq} with a diffusion coefficient of the form \( D(u,v) = m(\eps+u)^{m-1} \) with $m>1$ and $\eps\geq 0$, linear sensitivity $\chi(u, v) = \chi u$,  logistic cell growth $f(u, v) = au-bu^2$, no chemical signal production $g(u,v)=0$,  and chemical consumption $h(u, v) = u$, which leads to the following system: 
\begin{equation}
\label{main-perturbed-eq}
\begin{cases}
u_t = m\nabla\cdot \big((\eps+u)^{m-1}\nabla u\big) - \chi \nabla \cdot (u \nabla v) + u(a - b u),\quad & x \in \mathbb{R}^N, \,\, t>0, \\
v_t = \Delta v - uv, \quad & x \in \mathbb{R}^N,\,\,  t>0,\\
u(0,x) = u_0(x),\,\, v(0, x) = v_0(x), \quad & x\in \R^N
\end{cases}
\end{equation}
with initial condition satisfying

\begin{equation*}
u_0(\cdot)\in L^\infty(\R^N),\quad v_0(\cdot)\in W^{1,\infty}(\R^N).
\end{equation*}
Let us highlight that our framework allows both $u_0$ and $v_0$ to be non-integrable. This generality is important, as it lays the foundation for further studies, including the large-time propagation behavior of chemotaxis models in the whole space (see \cite{griette2023speed,hassan2025spreading,salako2018existence,shen2021can}).

We refer to system \eqref{main-perturbed-eq} as a chemotaxis model with nonlinear diffusion, a logistic source, and consumption of the chemical signal. When $\eps>0$, the diffusion term is non-degenerate. In contrast, for $\eps=0$, the diffusion becomes degenerate of porous medium type, and \eqref{main-perturbed-eq} reduces to
\begin{equation}
\label{main-eq}
\begin{cases}
u_t = \Delta u^m- \chi \nabla \cdot (u \nabla v) + u(a - b u),\quad & x \in \mathbb{R}^N,\,\,  t>0, \\
v_t = \Delta v - uv, \quad & x \in \mathbb{R}^N, t>0,\,\, \\
u(0,x) = u_0(x),\,\, v(0, x) = v_0(x), \quad & x\in \R^N.
\end{cases}
\end{equation}
We refer to \eqref{main-perturbed-eq} as a perturbed problem of \eqref{main-eq}. The study of chemotaxis models incorporating porous medium-type diffusion is motivated by the observation that the migration of cells is more accurately described by nonlinear diffusion, where cell mobility depends nonlinearly on the cell density \cite{szymanska2009nonlocal}. It is common in the literature (see, e.g., \cite{vazquez2007porous}) to approximate the porous medium equation via non-degenerate parabolic equations by passing to the limit $\eps\to0$.  The main objective of this paper is to establish the global solvability and boundedness of solutions
to the chemotaxis system \eqref{main-eq}. Our approach is to first prove the global existence of classical solutions to the perturbed problem \eqref{main-perturbed-eq}, and then pass to the limit  $\eps \to 0$ to obtain a weak solution of the degenerate system \eqref{main-eq}.

When $m=1$, \eqref{main-perturbed-eq} reduces to
\begin{equation}
\label{special-eq1}
\begin{cases}
u_t = \Delta u - \chi \nabla \cdot (u \nabla v) + u(a - b u),\quad & x \in \mathbb{R}^N,\,\, t>0, \\
v_t = \Delta v - uv, \quad & x \in \mathbb{R}^N, \,\, t>0,\\
u(0,x) = u_0(x),\,\, v(0, x) = v_0(x), \quad & x\in \R^N.
\end{cases}
\end{equation}
Global solvability and boundedness of the problem have been investigated in \cite{ hassan2024chemotaxis,  hassan2025spreading}. In \cite{hassan2024chemotaxis}, the authors of this paper established the global existence and boundedness of solutions to \eqref{special-eq1} under the condition
\begin{equation*}
|\chi|\cdot \|v_0\|_\infty <\max\left\{ {D^*_{\tau,N}, \,  b \cdot {C^*_{N}} }\right\} ,
\end{equation*}
where $C^*_N = \infty$ for $N=1, 2$  (see \cite{hassan2024chemotaxis} for the definitions of $D_{\tau,N}^*$ and $C_N^*$). Moreover, in \cite{hassan2025spreading}, the spreading properties of globally defined, bounded, positive solutions to \eqref{special-eq1} were further analyzed.  The reader is referred to \cite{chen2025porous, duan2010global,  liu2011coupled,xu2025porous}
for {studies on the global solvability of certain} coupled chemotaxis-fluid equations on $\mathbb{R}^2$ and $\mathbb{R}^3$.

Consider the following counterpart of \eqref{main-perturbed-eq} on a smooth bounded domain $\Omega\subset\R^N$,
\begin{equation}
\label{special-eq2}
\begin{cases}
u_t = m\nabla\cdot \big((\eps+u)^{m-1}\nabla u\big) - \chi \nabla \cdot (u \nabla v) + u(a - b u),\quad & x \in \Omega,\,\,\,\, \,t>0, \\
v_t = \Delta v - uv, \quad & x \in \Omega,\,\, \,\,\, t>0,\\
(\nabla (u+\eps)^m - \chi u\nabla v)\cdot \nu =\nabla v\cdot \nu = 0, \quad & x\in \partial\Omega,\,\, t>0,\\
u(0, x) = u_0(x),\,\, v(0, x) = v_0(x), \quad & x\in \Omega. 
\end{cases}
\end{equation}
When $\eps>0$ and $m>1$, global existence and boundedness of classical solutions of  \eqref{special-eq2}  have been studied in  \cite{marras2018boundedness,song2019new,wang2015higher, wang2014boundedness,wu2019boundedness, zhang2021boundedness}, etc.  It has been  shown that for any sufficiently smooth initial data, there exists a classical solution of \eqref{special-eq2},  which is global in time, and uniformly bounded for all times
(see \cite[Theorem 1.1]{song2019new}). 
When $\eps=0$ and $m>1$, global existence of weak solutions of \eqref{special-eq2} has been studied in \cite{huang2021two_species, jin2017boundedness, jin2019porous-stability, ye2022periodic}, etc., for the case $N=3$.  It has been  shown that for any  non-negative initial datum $u_0\in L^\infty(\Omega)\cap W^{1,2}(\Omega)$, $v_0\in C^{2}(\Bar{\Omega})$ and any $m>1$,  \eqref{special-eq2} with $\eps=0$ and $N=3$  admits a global weak solution, which is uniformly bounded for all time (see \cite[Theorem 1.1]{jin2017boundedness}). In contrast, in the case when $m=1$, {$N\ge 3$} and $\eps=0$, global existence has only been established under the assumption that the initial condition $v_0$ is sufficiently small in a certain sense. It remains unclear whether these smallness conditions are necessary.   {The reader is referred to \cite{chung2017porous-fluid, jin2018porous-fluid, jin2024porous-fluid, karuppusamy2024porous-fluid, tao2012global, tian2023porous-fluid, wang2023porous-fluid, xiang2023porous-fluid, yu2020porous-fluid}, among others, for studies on the global solvability of certain coupled chemotaxis-fluid equations with porous medium diffusion in two- and three-dimensional domains.}

To the best of our knowledge, all existing results on chemotaxis models with porous medium-type diffusion in unbounded domains consider only integrable initial data (see \cite{sugiyama2006global, sugiyama2007time, sugiyama2006decay}).  Moreover, the global existence of weak solutions both in bounded and unbounded domains, typically requires restrictions on either the diffusion exponent $m$ or the spatial dimension $N$ (see Remark~\ref{rk-1.2}). In this work, we establish for the first time the global existence of weak solutions for non-integrable initial data in unbounded domains. Our results hold without any restrictions on $m$ or $N$, thereby improving and extending existing results in the literature. 

\medskip

In the following, let us highlight the novelty of  our argument. One of the main contribution of the paper is to provide an approach that establishes a local $L^p$ estimates independent of $\eps$ and time for $u$ in \eqref{main-perturbed-eq} given any bounded initial data, which can be non-integrable.  

Note that, in  the case when the domain $\Omega$ is bounded, one can multiply $u^{p}$ on both sides of 
the first equation in \eqref{special-eq2}  to get,
\begin{align*}
\frac{1}{p+1}\frac{d}{dt}\int_{\Omega} u^{p+1}\leq {-mp\int_{\Omega} u^{p+m-2}|\nabla u|^2} 
 + |\chi|\int_{\Omega} u^{p} |\nabla u||\nabla v|+ \int_{\Omega}  au^{p+1}  - b \int_{\Omega}  u^{p+2}.
\end{align*}
From this, one can obtain that for some $C>0$ independent of $u$ and $v$,
\beq\lb{667}
\begin{aligned}
&\quad\,\int_{\R^N} u^{p+1}(t,x)+\frac{mp}{2}\int_0^t \int_{\Omega} u^{p+m-2}|\nabla u|^2+\frac{b}{2}\int_0^t\int_{\Omega}
u^{p+2}(s,x)\\
&\le \int_{\R^N} u^{p+1}(0,x)+C\int_0^t\int_{\Omega} |\nabla v|^{\frac{2(p+2)}{m}} +C\int_0^t\int_{\Omega} u(s,x).
\end{aligned}
\eeq
After treating $v$ carefully, the right-hand side can be bounded by the left-hand side plus the space-time $L^1$-norm of $u$  and some constant, which then implies the $L^p$-estimate for $u$. The proof proceeds from here using a Gr\"{o}nwall's type inequality. We refer readers to \cite{jin2017boundedness} for more details.


However, in our case, for non-integrable solutions in the whole domain, we can not compute the global $L^p$ norm of the solutions, and so we need to localize the problem in space.
One naive idea to localize the problem is to multiply the equation of $u$ by $u^{p}\varphi$ where $\varphi$ is a certain cut-off function. Consequently, we encounter an extra term of $\iint_{(0,T)\times\R^N} {} u^{p+m-1}|\nabla u||\nabla \varphi|dxdt$ which can be bounded by
\[
\iint_{(0,T)\times\R^N} {}u^{p+m-2}|\nabla u|^2\varphi dxdt+\iint_{(0,T)\times\R^N} {}u^{p+m}|\nabla \varphi|^2/\varphi dxdt.
\]
Although the second term involves just zero-order differentiation, the degree of $u$ becomes higher as $m$ increases. In particular, the degree is no less than $p+2$ when $m> 2$. It is not easy to control this term directly using the combination of the diffusion term and  
the logistic term on the left-hand side of \eqref{667}. The classical argument 
based on Gr\"{o}nwall's type inequalities
does not apply in this setting. 

To address this issue, we construct a specific exponentially decaying cut-off function, with a decay rate determined by the initial data and other relevant parameters, and use it to derive an inequality of the form \eqref{estimate-case2-eq3}. This inequality is crucial because it allows us to fully exploit all the favorable contributions -- namely, the diffusion term, the logistic term, and the regularity of the initial data (see Subsection \ref{S.3.3}). Our approach is close to a continuity-type argument, which, to the best of our knowledge, is new in the study of chemotaxis models.



\subsection{Definitions and main results}

In this subsection, we introduce the definitions of weak solutions of \eqref{main-perturbed-eq} with $\eps\ge 0$ and classical solutions of  \eqref{main-perturbed-eq} with $\eps>0$, 
 and state the main results of the paper.

\begin{defin}
\label{D.1} 
 Let $m>1$ and $T>0$, and let $u_0\in L^\infty(\R^N)$ and $v_0\in W^{1,\infty}(\R^N)$. A pair $(u, v)$ of non-negative functions defined
in $[0, T) \times \R^N$ is called a weak solution of \eqref{main-perturbed-eq} on $[0, T)$ if
\begin{itemize}
    \item[(1)]  $u\in L^2(0, T; L_{\rm loc}^2(\R^N))$,  $(\eps+u)^{m-1}\nabla u \in L^2(0,T; L_{\rm loc}^2(\R^N))$;

    \item[(2)] $v\in L^\infty(0,T; H^1_{\rm loc}(\R^N))$, $u\nabla v \in L^2(0,T; L^2_{\rm loc}(\R^N))$;

    \item[(3)]  For any continuously differentiable function $\psi$ with compact support in $[0,T)\times\R^N$, we have
    \begin{align*}
    &\quad\, \int_0^T \int_{\mathbb{R}^N} u \psi_t \, dx dt + \int_{\mathbb{R}^N} u_0(x) \psi(x, 0) \, dx\nonumber\\
&   = \int_0^T \int_{\mathbb{R}^N} \left( m(\eps+u)^{m-1}\nabla u \cdot  \nabla \psi - \chi u\nabla v\cdot \nabla \psi - au \psi + bu^2 \psi \right) \, dx dt
\end{align*}
and
\begin{align*}
    \int_0^T \int_{\mathbb{R}^N} v \psi_t \, dx dt + \int_{\mathbb{R}^N} v_0(x) \psi(x, 0) \, dx 
    = \int_0^T \int_{\mathbb{R}^N} \left[ \nabla v \cdot \nabla \psi + uv \psi \right] \, dx dt.
\end{align*}
\end{itemize}
\end{defin}

Our first main result is  on a priori estimate  of  classical  solutions  of \eqref{main-perturbed-eq}, which  is stated in the following theorem. We denote $\Omega_T:=[0,T]\times\R^N$.

\begin{tm}[A priori estimate]
\label{apriori-boundedness-thm}
Let $\eps\in(0,1)$, $m>1$ and $T>0$, and let $u_0\in L^\infty(\R^N)$ and $v_0\in W^{1,\infty}(\R^N)$ such that $u_0,v_0\geq 0$. Suppose that $(u_\eps, v_\eps)$ is a weak solution to \eqref{main-perturbed-eq} with initial data $(u_0, v_0)$, and they satisfy the equation in the classical sense for positive times. Then
\begin{itemize}
\item[(1)] ($L_{\rm loc}^p$ a priori estimate). 
  For any $p\geq m$ 
  and $t\in [0,T]$,
\begin{equation}
\label{main-a-priori-est1}
u_\eps(t,\cdot)\in L_{\rm loc}^{p+1}(\R^N), \,
 e^{-T/2}  (\eps+u_\eps)^{\frac{p+m-2}{2}}\nabla  u_\eps\in  L^2_{\rm loc}(\Omega_T),\,
v_\eps\in L^\infty(\Omega_T)\cap  W^{1,\infty}(\Omega_T)
\end{equation}
with a bound depending only on $m,|\chi|,a,b,N,p$, $\|u_0\|_\infty$,  $\|v_0\|_{W^{1,\infty}}$  and the diameter of the local spatial domain (but independent of $T$). 


\item[(2)]  ($L^\infty$ a priori estimate)  There exists $C$ depending only on $m,|\chi|,a,b,N,p$, $\|u_0\|_\infty$,  $\|v_0\|_{W^{1,\infty}}$ 
(but independent of $T$) 
such that
\begin{equation}
\label{main-a-priori-est2}
\|\nabla v_\eps\|_{L^\infty(\Omega_T)},\quad  \|u_\eps\|_{L^\infty(\Omega_T)} \le C.
\end{equation}
\end{itemize}
\end{tm}

In the following, we say  $v\in C^{1+\alpha,2+\alpha}$ if  $v$ is $C^{1+\alpha}$ in time and $C^{2+\alpha}$ in space; and  $u \in C^{\alpha}$ $([0,T)\times\R^N)\cap  C^{1+\alpha, 2+\alpha}((0,T)\times\R^N )$ if  $u$ is H\"{o}lder continuous in both time and space up to the initial, and $u$ is $C^{1+\alpha}$  in time and $C^{2+\alpha}$ in space for positive time. 

\begin{defin}
\label{D.2} 
Let $u_0$ be uniformly {$C^{1+\alpha}$} and $v_0$ be uniformly $C^{2+\alpha}$. A pair $(u, v)$ of non-negative functions defined
in $ [0, T)\times\R^N$ is called a classical solution of \eqref{main-perturbed-eq} with $\eps>0$ on $[0, T)$ if 
\begin{itemize}
    \item[(1)]  $v$ is uniformly $C^{1+\alpha,2+\alpha}$ and $u$ is uniformly in $C^{\alpha}$ $([0,T)\times\R^N)\cap  C^{1+\alpha, 2+\alpha}((0,T)\times\R^N )$;

    \item[(2)] $u(0,\cdot)=u_0$, $v(0,\cdot)=v_0$, and \eqref{main-perturbed-eq} is satisfied in the classical sense in $(0,T)\times \R^N$.
\end{itemize}
\end{defin}

As a corollary of these a priori estimates, we obtain global existence and boundedness of classical solutions of the perturbed problem
\eqref{main-perturbed-eq} with $\eps\in (0,1)$.

\begin{prop}[Classical solutions of \eqref{main-perturbed-eq}]
\label{main-perturbed-thm}
Under the assumptions of Theorem \ref{apriori-boundedness-thm}, further assume that $u_0$ is uniformly $C^{1+\alpha}$, and $v_0$ is uniformly $C^{2+\alpha}$. 
Then for each $\eps\in (0,1)$, there exists a  unique global classical solution  $(u_\eps$,  $v_\eps)$ of \eqref{main-perturbed-eq}  with initial condition 
$u_0,v_0$.

Moreover, the regularity properties presented in Theorem \ref{apriori-boundedness-thm} hold the same.
\end{prop}

Our last main result is  on the global  existence  and boundedness of weak solutions of \eqref{main-eq}.

\begin{tm}[Weak solutions of \eqref{main-eq}]
\label{main-thm}
Under the assumptions of Theorem \ref{apriori-boundedness-thm}, there exists a global weak solution $(u,v)$ to \eqref{main-eq} with initial data $(u_0,v_0)$. 

Moreover, the regularity properties presented in Theorem \ref{apriori-boundedness-thm} hold the same for $(u,v)$.
\end{tm}





\begin{rk}\lb{rk-1.2}
Several remarks are in order concerning our results. 

\item[(1)] (Related results for global existence on $\R^N$). Consider the following chemotaxis system with porous medium-type diffusion and linear production of the chemical signal:
\begin{equation}
\begin{cases}
u_t = \Delta  u^m - \nabla \cdot (u \nabla v) + u(a - b u), & x \in \mathbb{R}^N, \\
v_t = \Delta v - v + u, & x \in \mathbb{R}^N, \\
u(0, x) = u_0(x),\quad v(0, x) = v_0(x). &
\end{cases}
\end{equation}
When $a=b=0$, the global existence of weak solutions for integrable initial data, specifically $u_0\in L^1\cap L^\infty$ and $v_0\in L^1\cap H^1\cap W^{1,\infty}(\R^N)$ was established in \cite{sugiyama2007time} under certain restrictions on the exponent $m$ (the result also apply to the case of $a, b>0$ and their method will still require a restriction on $m$). We proved the global existence of weak solution for the case of consumption without any restriction on $m.$

\item[(2)] (Related results on bounded domain). In the setting of a bounded domain with Neumann boundary conditions and chemical signal consumption (equation \eqref{special-eq2}), the author in \cite{jin2017boundedness} proved global existence of weak solutions under the dimension constraint $N=3$. We note that our methods also extend to bounded domains and, importantly, remove the dimensional restriction, thereby generalizing the result in \cite{jin2017boundedness} to arbitrary spatial dimensions.

\item[(3)] (Uniqueness and regularity). 
In this paper, we only discuss the uniqueness of classical solutions to the perturbed problem. The uniqueness of the weak solution to \eqref{main-eq}, under certain conditions, will be studied in \cite{hassan2025regularity},  the second part of this series. 
In that work, we will also establish H\"{o}lder regularity results for solutions to both \eqref{main-perturbed-eq} and \eqref{main-eq}, with estimates that are uniform in $\eps$.
 
\item[(4)] (Novelty of our work). In this paper, we investigate the global existence and boundedness of solutions to a chemotaxis system with porous medium diffusion, logistic source, and chemical signal consumption on the whole space $\R^N$. To the best of our knowledge, this is the first work to establish global existence of weak solutions for such a model on 
$\R^N$ without imposing restrictions on the spatial dimension $N$ or the diffusion exponent $m>1$. Our results cover both integrable and non-integrable initial data. Notably, the global solvability for non-integrable initial data in the context of chemotaxis models with porous medium diffusion appears to be new.

\item[(5)] (Integrable data on $\bbR^N$). In the special case where $u_0 \in L^1 \cap L^\infty$ and $v_0 \in L^1 \cap W^{1,\infty}$, the existence of a global weak solution to \eqref{main-perturbed-eq} can be established more easily. In fact, one may simply take the cut-off function to be $1$, and many of the more delicate estimates become unnecessary. However, on unbounded domains, the potential growth of the $L^1$ norm over time makes it trickier to obtain uniform-in-time $L^\infty$ bounds. Nevertheless, our argument fully covers this case.
\end{rk} 

The rest of the paper is organized as follows. In Section 2, we present several preliminary lemmas that will be used throughout the subsequent analysis. Section 3 is devoted to establishing local $L^p$-estimates for the perturbed problem and proving part (1) of Theorem \ref{apriori-boundedness-thm}. In Section 4, we complete the proof of Theorem \ref{apriori-boundedness-thm} by establishing part (2). Section 5 is concerned with the global well-posedness of weak and classical solutions to system \eqref{main-perturbed-eq}, where we prove Proposition~\ref{main-perturbed-thm} and Theorem~\ref{main-thm}. We
prove Lemmas \ref{maximal-regularity-lm} and \ref{v-bound-lm} in the Appendix.

\subsection{Acknowledgements}
Our research was supported in part by Simons Foundation Travel Support MPS-TSM-00007305 (YPZ) and NSF CAREER grant DMS-2440215 (YPZ, ZH).

\section{Preliminary}\lb{ss.2.1}

In this section, we present some preliminary lemmas that will be used in the rest of the paper. These include $L^p-L^q$ estimates for the analytic semigroup generated by $\Delta-I$ on $L^p$,  a class of useful cut-off functions together with their properties,
and  a  maximal regularity for parabolic equations on
    $\R^N$.

We start with some basic properties of the analytic semigroup $T_p(t)$  generated by $\Delta-I$ on $L^p(\mathbb{R}^N)$ $(p\ge 1$). This is defined by
\begin{equation}
\label{semigroup-eq}
(T_p(t)u)(x)=e^{-t}(G(t, \cdot)\ast u)(x)= \int_{\R^{N}}e^{-t}G(t, x-y)u(y)dy
\end{equation}
for every $u\in L^p(\mathbb{R}^N)$, $t> 0$,  and $x\in\R^N$, where  $G(t,x)$ is the heat kernel defined by
\begin{equation}
\label{heat-kernel}
G(t,x)={(4\pi t)^{-\frac{N}{2}}}e^{-\frac{|x|^{2}}{4t}}.
\end{equation}
It follows from  the $L^p-L^q$
estimates for the convolution product   that there is  $C_{p, q}>0$ ($1\le p< q\le \infty$) such that
\begin{equation}\label{Lp Estimates-2}
 \| T_p(t)u \|_{L^{q}(\mathbb{R}^{N})}\leq C_{p,q} e^{-t} t^{-(\frac{1}{p}-\frac{1}{q})\frac{N}{2}}\|u\|_{L^{p}(\mathbb{R}^{N})},
\end{equation}
and
\begin{equation}\label{Lp Estimates-3}
 \| \nabla T_p(t)u \|_{L^{q}(\mathbb{R}^{n})}\leq C_{p,q} e^{-t}t^{-\frac{1}{2}-(\frac{1}{p}-\frac{1}{q})\frac{N}{2}}\|u\|_{L^{p}(\mathbb{R}^{N})},
\end{equation}
 for every $u\in L^p(\R^N)$ and $t>0$.
Note that, if $u\in L^p(\R^N)\cap L^\infty(\R^N)$, then $T_p(t)u\in L^\infty(\R^N)$ and
\begin{equation}
\label{L-infty- Estimates-1}
 \| T_p(t)u \|_{L^{\infty}(\mathbb{R}^{N})}\leq  e^{-t} \|u\|_{L^{\infty}(\mathbb{R}^{N})}.
\end{equation}

Now, we present a useful exponential decaying function. Subsequently, we denote
\[
B_r(x) := \{ y \in \mathbb{R}^N : |x - y| < r \}, \quad \text{and} \quad B_r := B_r(0).
\]
Also, $|B_r|$ represents the Lebesgue measure of N-dimensional ball of radius $r$.

Let us take a smooth decreasing function $f$ on $\R$ such that
\[
f(r)=1\quad\text{when }r\leq N\quad\text{ and }\quad f(r)=2^{-1} e^{N+1-r}\quad\text{when }r\geq N+1.
\]
For each fixed $\kappa\in (0,1)$ and for some $\gamma\in (0,1)$, define
\beq\lb{varphi}
\varphi(x):=\varphi_\kappa(x)=f(\gamma\kappa|x|).
\eeq

\begin{lem}
\label{psi-lm}
There exist dimensional constants $C,\gamma>0$ such that for any $\kappa\in (0,1)$, $\varphi$ from 
 \eqref{varphi} satisfies for all $x\in\R^N$,
\begin{equation}\label{phi1}
 0<\varphi(x)\leq 1,\quad |\nabla \varphi(x)|\le \k\, \varphi(x),\quad |D^2 \varphi(x)|\le \k^2\,\varphi(x),
\end{equation}
\beq\lb{phi2}
\varphi(x)\leq C\varphi(y)\quad \text{ whenever }\quad |x-y|\leq \kappa^{-1},
\eeq
and
\beq\lb{phi3}
\kappa^{N}\int_{\R^N}\varphi(x) dx,\quad \sum_{\kappa z\in \Z^N}\varphi(z)\leq C.
\eeq

\end{lem}
\begin{proof}
The proof of \eqref{phi1} and \eqref{phi2} is given by direct computations when $\gamma>0$ is sufficiently small depending only on $N$. Moreover, direct computation yields
\[
\kappa^{N}\int_{\R^N}\varphi(x) dx =\kappa^{N}\int_{\R^N} f(\gamma\kappa |x|)dx={\frac{2\pi^{\frac{N}{2}}}{\Gamma(\frac{N}{2})}}\int_{\R^+} f(\gamma r)r^{N-1}dr\leq C,
\]
where  $\frac{2\pi^{\frac{N}{2}}}{\Gamma(\frac{N}{2})}$ is the surface area of an $N$-dimensional unit sphere and $\Gamma$  is the gamma function.
Note that the collection of balls $\{B_{N/\kappa}(z)\,|\, \kappa z\in\Z^N \}$ covers the whole space $\R^N$, and by \eqref{phi2}, for any $z$,
\[
\varphi(z)\leq C^N\varphi(x)\quad \text{ if }x\in B_{\kappa^{-1}N}(z).
\]
Thus, it follows that for some dimensional constant $C$,
\[
\sum_{\kappa z\in \Z^N}\varphi(z)\leq C\kappa^N \sum_{\kappa z\in \Z^N} \int_{B_{\kappa^{-1}N}(z)}\varphi(x)dx\leq C\kappa^{N}\int_{\R^N}\varphi(x) dx.
\]
\end{proof}

The next lemma will be used in several places of the paper.
\begin{lem}
\label{I1-lm}
Assume that  $N\geq 3$ and let  $2^*:=\frac{2N}{N-2}>2$ be the Sobolev conjugate exponent of $2$.  Then  there exists a dimensional constant $C$ such that for any $\delta,\kappa>0$ and $r>1$ we have
\begin{align}
\lb{poin}
\kappa^2\int_{\R^N}u^{2r}\varphi^2
&\leq  \delta\|\nabla(u^{r}\varphi)\|_{2}^2+C\kappa^{2+2\theta_r}\delta^{-\theta_r}\left[\int_{\R^N}u^{r+1}\varphi^{q_r}dx\right]^{2/q_r}, 
\end{align}
where $u=u(x)$ is any function such that $u^r$ is  locally uniform finite in $W^{1,2}$-space, and $\varphi$ is from Lemma \ref{psi-lm}, and
\[
q_r:=\frac{r+1}{r}\in (1, 2)\quad {\rm and}\quad \theta_r:=\frac{N(r-1)}{2(r+1)}>0.
\]
\end{lem}

In fact, one could replace $u\varphi$ with a single function $\tilde{u}$ and assume $\tilde{u}^r \in W^{1,2} \cap L^{\frac{r+1}{r}}$,  without imposing separate assumptions on $u^r$ and $\varphi$. However, we choose to state the lemma in its current form for future applicability.

\begin{proof} Note that $u^r\varphi\in W^{1,2}(\R^N)$.
Then by Gagliardo-Nirenberg-Sobolev inequality, there is a dimensional constant $\widetilde C$ such that for any $r>1$,
\begin{align}\label{l2-est11}
  \|u^{r}\varphi\|_{2^*}  
&\le \widetilde C\|\nabla(u^{r}\varphi)\|_{2}.
\end{align}
By the interpolation inequality of $L^p$ spaces, we have
\[
\begin{aligned}
    \|u^{r}\varphi\|_{2} 
     \le  \|u^{r}\varphi\|_{q_r}^{1-\vartheta}\|u^{r}\varphi\|_{2^*}^{\vartheta}.
\end{aligned}
\]
Here direct computation yields
\[q_r:=\frac{r+1}{r}\in (1, 2),\quad \vartheta:=\frac{2^*(2-q_r)}{2(2^*-q_r)}=\frac{N(r-1)}{r(N+2)-N+2}\in (0,1),\]
and so
\[
2+2\theta_r:=\frac{2}{1-\vartheta}=\frac{r(N+2)-N+2}{r+1}\quad\text{and so}\quad \theta_r=\frac{N(r-1)}{2(r+1)}>0.
\]
Hence, by Young's inequality and \eqref{l2-est11}, for any $\delta$,
\begin{align*}
\kappa^2\int_{\R^N}u^{2r}\varphi^2&\leq   \frac{\delta}{\widetilde C}\|u^{r}\varphi\|_{2^*}^2+\widetilde C^{\frac{\theta}{1-\theta}}\kappa^{\frac{2}{1-\vartheta}}\delta^{-\frac{\vartheta}{1-\vartheta}}\|u^{r}\varphi\|_{q_r}^2\nonumber\\
&\leq  \delta\|\nabla(u^r\varphi)\|_{2}^2+ C \kappa^{2+2\theta_r}\delta^{-\theta_r}\left[\int_{\R^N}u^{r+1}\varphi^{q_r}dx\right]^{2/q_r},
\end{align*}
where $C=\widetilde C^{\frac{\theta}{1-\theta}}$.
The lemma is thus proved.
\end{proof}

We end the section by  presenting a maximal regularity for some parabolic equations on $\R^N$. The proof is postponed to the Appendix.
    
    \begin{lem}
    \label{maximal-regularity-lm}
        Let    $v_0\in W^{1,\gamma}(\R^N)\cap L^\infty(\R^N)$.  There exists $C_{\gamma,N}$  such that for any 
  $T\in (0,\infty)$, if  $g \in L^\gamma((0,T), L^\gamma(\R^N))$ and $v(\cdot,\cdot)\in W^{1,\gamma}((0,T),L^{\gamma}(\R^N))\cap L^{\gamma}((0,T), W^{2,\gamma}(\R^N))$ solves  the following initial boundary value problem,
    \begin{equation}
    \label{pdelaplace}
    \begin{cases}
     v_t =\Delta v -  v + g,\quad x\in \R^N,\,\,  0<t<T\cr
    v(0,x) = v_0(x),\quad x\in \R^N,
    \end{cases}
    \end{equation}
    then 
    \beq\lb{maximal-regularity-eq1-1}
    \begin{aligned}
   &\quad\, \int_{0}^T \int_{\R^N}e^{ t}\left( |v (t,x)|^\gamma +|\nabla v(t,x)|^\gamma +|\Delta v(t,x)|^\gamma\right)dxdt\\
& \le C_{\gamma,N}\left[ \int_{0}^T \int_{\R^N}e^{ t}|g(t,x)|^{\gamma}dx dt+ T\left( \|v_0(\cdot)\|^{\gamma}_{L^\gamma(\R^N)}+\|\nabla v_0(\cdot)\|_{L^\gamma(\R^N)}^\gamma\right)\right] .
    \end{aligned}
    \eeq
    \end{lem}

\section{$L_{\rm loc}^{p}$ a priori estimate and the proof of Theorem \ref{apriori-boundedness-thm}(1)}\lb{S.3}

In this section, we shall give $L_{\rm loc}^{p+1}(\R^N)$  a priori estimate of solutions of the perturbed problem \eqref{main-perturbed-eq} on any finite time interval $[0,T]$, and prove  Theorem \ref{apriori-boundedness-thm}(1). 


Let $0<\eps<1$ and $0<\kappa<1$, and take $p$ such that
\begin{equation}
\label{m-eq3}
 p\geq m.
\end{equation} 
Throughout this section, the constants $c\in (0,1)$ and $C\geq 1$ only depend on   $m,\chi,a,b$, $N,p$ and  $\|v_0\|_{W^{1,\infty}}$ and $\|u_0\|_{L^{\infty}}$, {and they may be different at different places},  unless otherwise stated. 
By $C_1\lesssim C_2$,
we mean that there exists $C>0$  such that $C_1\leq CC_2$. When we say that a constant depends on $C$, we mean that it might depend on $m,\chi,a,b$, $N$  and  $\|v_0\|_{W^{1,\infty}}$ and $\|u_0\|_{L^{\infty}}$.
We emphasize that the constants $c,C$ will be always independent of   $\eps$,  $\kappa$, and $T$!

Recall that
\[
\Omega_T:=[0,T]\times\R^N.
\]
We also denote $\Omega_t:=[0,t]\times\R^N$ for any $t\in (0,T]$. 
Suppose that $(u(t,x),v(t,x))$ solve the perturbed problem \eqref{main-perturbed-eq} {\rm  with $\eps\in (0,1)$ for $t\in (0,T]$} in the classical sense and let $\varphi$ be from Lemma \ref{psi-lm} with parameter $\kappa$.
For any given $x_0\in\R^N$ and $r>1$, define
\begin{equation*}
{Z_{r,x_0}(t)}:= \iint_{ \Omega_{t}} e^{-(t-s)} (u+\eps)^r(s,x) \varphi^2(x-x_0)dxds\quad\text{and}\quad
{Z_r(t)}:=\sup_{x_0\in\R^N}Z_{r,x_0}(t).
\end{equation*}
We start with two lemmas in next subsection.

\subsection{Lemmas}

In this subsection, we present two lemmas to be used in estimating the local $L^p$-norm of $u$
and proving Theorem \ref{apriori-boundedness-thm}(1). 

\begin{lem}
\label{I3-lm}
For any fixed $0<\delta<1$ and $r'>r> 1$ ,  there exist $c_r$ depending only on $r$ and $N$, and $C_\delta$ depending only on $C,r,r'$ and $\delta$ such that for all $\kappa\in (0,c_r)$ we have
\begin{equation}
\label{gradient-v-eq}
\iint_{{ \Omega_{t}}}  e^{-(t-s)}|\nabla v|^{2r}\varphi^2\leq \delta { Z_{r'} (t)}+C_\delta\,\kappa^{-N}\quad\forall\, t\in {(0,T]}.
\end{equation}
\end{lem}

\begin{proof}
First,  note that
\beq\lb{4.19}
\iint_{\Omega_t}{e^{-(t-s)}} |\nabla v|^{2r}\varphi^2\lesssim \iint_{\Omega_t} { e^{-(t-s)}} |\nabla (v\varphi^{\frac1r})|^{2r}+\kappa^{2r}\iint_{\Omega_t} { e^{-(t-s)}} v^{2r}\varphi^2.
\eeq
Denoting $\varphi_1: = \varphi^\frac1r$, since $v$ is bounded by $\|v_0\|_\infty$ and $\varphi$ is bounded by $1$, we have
\beq\lb{4.20}
\|v \varphi_1\|_{2r}^2\leq \|v\varphi_1\|_\infty  \|v\varphi_1\|_{r}\lesssim \|v\varphi_1\|_r.
\eeq
By the Gagliardo-Nirenberg inequality,
$$
\|\nabla (v \varphi_1)\|_{2r} \lesssim \|v\varphi_1\|_\infty ^{\frac12} \|D^2 (v\varphi_1)\|^{\frac12}_{r}\lesssim \|D^2 (v\varphi_1)\|^{\frac12}_{r}.
$$
By \cite[Theorem 9.11]{gilbarg1998elliptic},
\beq\lb{4.5}
\|\nabla (v \varphi_1)\|_{2r} \lesssim \|D^2 (v\varphi_1)\|^{\frac12}_{r}\lesssim \|\Delta (v\varphi_1)\|^{\frac12}_{r}+\|v\varphi_1\|_r^{\frac{1}{2} }.
\eeq

Next, note that
\begin{align}
(v\varphi_1)_t& =\Delta(v\varphi_1)+\left(-2\nabla v\cdot \nabla\varphi_1 - v\Delta \varphi_1 - uv \varphi_1\right)\nonumber\\
&=\Delta(v\varphi_1)-v\varphi_1 +\left(v\varphi_1-2\nabla v\cdot \nabla\varphi_1 - v\Delta \varphi_1 - uv \varphi_1\right)
\end{align}
 Lemma \ref{maximal-regularity-lm} then  yields 
\begin{align*}
&\int_{0}^t{e^{-(t-s)}}\Big(\|\Delta(v\varphi_1)\|^{r}_{r}+ \|\nabla(v\varphi_1)\|^{r}_{r}+\| v\varphi_1\|^{r}_{r}\Big)\\ &\leq C\int_{0}^t \int_{\R^N} {e^{-(t-s)}} |v\varphi_1 - 2\nabla v\nabla\varphi_1 - v\Delta\varphi_1-u v\varphi_1|^{r} \, dx ds  + C{ t e^{-t} \|v_0\varphi_1\|_{W^{1,r}}^r}\nonumber\\
   &\leq C \|\varphi\|_1+ C{\int_{0}^t \int_{\R^N} e^{-(t-s)}}\big[ \kappa^r |\nabla (v \varphi_1 )|^r + |u \varphi_1|^r \big]dxds,
\end{align*} 
{where $C$ only depends on $r$  and $N$,  and is independent of $T$}.
 Therefore, if $\kappa$ is sufficiently small depending only on $r$ and $N$, by \eqref{phi3}, we get
 \[
\int_{0}^t e^{-(t-s)} \Big(\|\Delta (v\varphi_1)\|^{r}_{r}+\| v\varphi_1\|^{r}_{r}\Big) \lesssim \kappa^{-N}+ \int_{0}^t\int_{\R^N}  e^{-(t-s)} |u \varphi_1|^r,
 \]
which, combining with \eqref{4.19}, \eqref{4.20} and \eqref{4.5}, implies that
\beq\lb{4.3}
\iint_{\Omega_t} e^{-(t-s)} |\nabla v|^{2r}\varphi^2 +\iint_{\Omega_t}e^{-(t-s)} |D^2(v\varphi_1)|^r\lesssim \kappa^{-N}+ \iint_{\Omega_t} 
e^{-(t-s)} u^r \varphi.
\eeq

In the following, we derive a connection between $Z_{r'}(t)$ 
and $\iint_{\Omega_t} e^{-(t-s)} u^r \varphi$. 
Note that   the collection of balls $\{B_{N/\kappa}(z)\,|\, \kappa z\in\Z^N \}$ covers the whole space. Hence, by \eqref{phi2} and  using that $\varphi(x-z)=1$ when $x\in B_{N/\kappa}(z)$ by the definition of $\varphi$,  we have
\begin{align*}
\iint_{\Omega_t} e^{-(t-s)} u^r\varphi&\lesssim \sum_{\kappa z\in \Z^N}\varphi(z)\int_0^t e^{-(t-s)}\int_{B_{{N}/{\kappa}}(z)} u^r(s,x)dxds\nonumber\\
&\lesssim \sum_{\kappa z\in \Z^N}\varphi(z)\int_0^t e^{-(t-s)}\int_{B_{{N}/{\kappa}}(z)} u^r(s,x)\varphi^2(x-z)dxds\nonumber\\
&\lesssim  \sup_{x_0\in\R^N} \iint_{\Omega_t} e^{-(t-s)}u^r(s,x)\varphi^2 (x-x_0)dxds \nonumber\\
&\lesssim  \sup_{x_0\in\R^N}\Big(\delta \iint_{\Omega_t} e^{-(t-s)}u^{r'}(s,x)\varphi^2(x-x_0)dxds+C_{\delta,r,r'}\int_{\R^N} \varphi^2(x-x_0)dx\Big)\nonumber\\
&\lesssim \delta Z_{r'}(t)+C_{\delta,r,r'}\,\kappa^{-N},
\end{align*}
where in the third inequality, we used that 
\[
\sum_{\kappa z\in\Z^N}\varphi(z)\lesssim \kappa^N\int_{\R^N}\varphi(z)dz\lesssim 1
\]
by \eqref{phi3}.
The lemma is thus proved.
\end{proof}

\begin{lem}
\label{v-bound-lm} 
For any $p>N$ and $C_1>0$, if
\begin{equation*}
\sup_{t\in [0,T],x_0\in\R^N}\int_{B(x_0,1)} u^p(t,x)dx\leq C_1,
\end{equation*}
then there is a constant  $C$ depending only on $p$, $N$,  $C_1$, and $ \|v_0\|_{W^{1,\infty}}$ such that 
\begin{equation*}
\sup_{t\in [0,T],x\in\R^N}|\nabla v(t,x)|<C.
\end{equation*}
\end{lem}
The result can be found in \cite[Theorem 3.1, Chapter V]{ladyzhenskaya1968linear}. {For completeness, we provide a more direct proof in the appendix.}

Now, we proceed to prove local $L^{p+1}$ estimate for $u$ in next three subsections. We recall that the constants $c$ and $C$ below might depend on the data ($m,N$, etc.) and $p$, but are independent of $\eps,\kappa$ and $T$.
Multiplying the first equation of \eqref{main-perturbed-eq} by $(u+\eps)^{p}\,\varphi^2$ and integrating over $\R^N$  yields
\begin{align}
\label{new-new-lp-eq1}
&\quad\, \frac{1}{p+1}\frac{d}{dt}\int_{\R^N} (u+\eps)^{p+1}\varphi^2  dx\nonumber \\
&= \int_{\R^N} (u+\eps)^{p}\varphi^2\nabla\cdot\left[m(u+\eps)^{m-1}\nabla u - \chi u\nabla v\right] + (u+\eps)^{p}\varphi^2(au - bu^2)\nonumber \\
&\leq  -c\int_{\R^N} (u+\eps)^{p+m-2}|\nabla u|^2\varphi^2 +C\int_{\R^N}(u+\eps)^{p+m-1}|\nabla u||\nabla\varphi^2| \nonumber\\
 &\quad+ C\int_{\R^N} (u+\eps)^{p-1}u |\nabla u||\nabla v|\,\varphi^2+ C\int_{\R^N} (u+\eps)^{p}u |\nabla v||\nabla \varphi^2|+\int_{\R^N}(u+\eps)^{p}(au - bu^2)\varphi^2\nonumber\\
&\leq \underbrace{-c\int_{\R^N}  (u+\eps)^{m+p-2}|\nabla u|^2\varphi^2}_{-I_1} +\underbrace{C\kappa\int_{\R^N}(u+\eps)^{p+m-1}|\nabla u|\varphi^2}_{\kappa I_2}\nonumber \\
 &\quad+\underbrace{ C\int_{\R^N} (u+\eps)^{p} |\nabla u||\nabla v|\,\varphi^2}_{I_3}+ \underbrace{C\kappa\int_{\R^N} (u+\eps)^{p+1}|\nabla v| \varphi^2}_{\kappa I_4} +\underbrace{ \int_{\R^N}  (u+\eps)^p(au-b u^{2})\varphi^2}_{I_5},
\end{align}
where, in the second inequality, we used that $|\nabla \varphi|\leq\kappa\varphi$.

To prove Theorem \ref{apriori-boundedness-thm}(1),  it is essential to provide proper estimates for $I_1$, $I_2$, $I_3$,  $I_4$, and $I_5$. We point out that $I_1$ and $I_2$ depend on $m$, and we will estimate them  differently for the case
$1<m\le 2$ and $m>2$. In the following, we will first  estimate $I_3$, $I_4$ and $I_5$ in Subsection \ref{I3-I4-I5-subsec}. 
And then we estimate $I_2$ and prove Theorem  \ref{apriori-boundedness-thm}(1)  for the cases $1<m\le 2$ and $m>2$ in Subsections \ref{m-small-subsec} and
\ref{m-large-subsec}, respectively.

\subsection{Estimates for $I_3$, $I_4$ and $I_5$}
\label{I3-I4-I5-subsec}
First, we  estimate {$I_3(t)$}. For simplicity of notations, let us write
\[
u_+:=u+\eps\geq\eps.
\]
Since $\nabla u=\nabla u_+$, by Young's inequality, for any $0<\delta<1$, 
\begin{align}
\label{I3-eq1}
I_3(t)&\le C \int_{\R^N} u_+^p |\nabla u_+| (|\nabla v|+1)\varphi^2 dx\quad \nonumber\\
&\leq \delta \int_{\R^N} u_+^{p+m-2}|\nabla u_+|^2\varphi^2 +C\delta^{-1}\int_{\R^N}
u_+^{p+2-m}(|\nabla v|^2+1)\varphi^2\nonumber\\
&\leq \delta\int_{\R^N} u_+^{p+m-2}|\nabla u_+|^2\varphi^2 +\delta\int_{\R^N}u_+^{p+2}\varphi^2+C_\delta\int_{\R^N} |\nabla v|^{\frac{2(p+2)}{m}}\varphi^2+C_\delta\kappa^{-N},
\end{align}
where we used that $p+2-m> 0$ by \eqref{m-eq3}.

Next, we consider $\kappa I_4(t)$. By Young's inequality,
\begin{align}
\label{I4-eq1}
\kappa I_4(t)&\lesssim \kappa \int_{\R^N} u_+^{p+1}|\nabla v|\varphi^2\lesssim \kappa \int_{\R^N} u_+^{p+2}\varphi^2 +\kappa\int_{\R^N} |\nabla v|^{p+2}\varphi^2.
\end{align}

Finally, since $\eps\in (0,1)$, there exists $C>0$ depending only on $a,b$ and $p$,
\[
au-bu\leq (a+2b\eps)(u+\eps)-b(u+\eps)^2\leq -\frac{1}{p+1}u_+-Cu_+^2+C.
\]
By Lemma \ref{psi-lm}, it  is  clear that
\begin{equation}
\label{I5-eq1}
I_5\le -\frac{1}{p+1}\int_{\R^N} u_+^{p+1}\varphi^2dx -c\int_{\R^N} u_+^{p+2}\varphi^2 dx+C\kappa^{-N}.
\end{equation}
Here the term $\frac{1}{p+1}\int_{\R^N} u_+^{p+1}\varphi^2dx$ on the right-hand side will be used to have uniform-in-time bounds.

\subsection{Proof of Theorem  \ref{apriori-boundedness-thm}(1)  for the case $1<m\le 2$}
\label{m-small-subsec}

In this subsection, we consider the case of $1<m\le 2$,  and prove Theorem \ref{apriori-boundedness-thm}(1). This case is much simpler than the case when $m>2$, which is because for $m\leq 2$, the logistic term is strong enough to bound the integral of $u^{p+m}$. 

\begin{proof}[Proof of Theorem \ref{apriori-boundedness-thm}(1)  for the case $1<m\le 2$]
 Let $p \geq m$. We begin by estimating $\kappa I_2$. Since
\begin{equation*}
p+m\le p+2\quad {\rm for}\quad 1<m\le 2,
\end{equation*}
we have
$$
u_+^{p+m}\varphi^2\lesssim u_+^{p+2}\varphi^2+\varphi^2.
$$
By Young's inequality, we have for any $\delta\in (0,1)$ and  $0<\kappa<\delta$,
\begin{align}
\label{I2-eq1}
\kappa I_2&= C \kappa\int_{\R^N}u_+^{p+m-1}|\nabla u_+|\varphi^2
\lesssim {\delta}\int_{\R^N} u_+^{p+m-2}|\nabla u_+|^2 \varphi^2 + \delta^{-1}\kappa^2\int_{\R^N}u_+^{p+m}\varphi^2\nonumber\\
&\lesssim {\delta}\int_{\R^N} u_+^{p+m-2} |\nabla u_+|^2 \varphi^2dx + \delta \int_{\R^N}u_+^{p+2}\varphi^2 dx+C_\delta\kappa^{-N},
\end{align}
where in the last inequality, we also used Lemma \ref{psi-lm}.

After taking $\delta$  to be sufficiently small (then $\kappa<\delta$ is also small),  by \eqref{new-new-lp-eq1},  \eqref{I3-eq1}, \eqref{I4-eq1}, \eqref{I5-eq1} and \eqref{I2-eq1}, we have
\begin{align*}
\frac{d}{dt}\int_{\R^N} u_+^{p+1}\varphi^2  dx
&\le -c \int_{\R^N} u_+^{p+m-2}|\nabla u_+|^2\varphi^2 -c\int_{\R^N} u_+^{p+2}\varphi^2 -\int_{\R^N} u_+^{p+1}\varphi^2 \nonumber\\
&\quad +  C_\delta\int_{\R^N} |\nabla v|^{\frac{2(p+2)}{m}}\varphi^2+  \kappa\int_{\R^N} |\nabla v|^{p+2}\varphi^2+C_\delta \kappa^{-N}.
\end{align*}
Multiplying $e^{t}$ to both sides and integrating in time, this implies that
\begin{align}
\label{aux-lp-eq2}
&\int_{\R^N}u_+^{p+1}(t,x)\varphi^2(x)dx+c\iint_{\Omega_{t}} e^{-(t-s)}u_+^{p+m-2}|\nabla u_+|^2\varphi^2dxds+c\iint_{\Omega_{t}}e^{-(t-s)}u_+^{p+2}\varphi^2dxds\nonumber\\
&\le \int_{\R^N} u_+^{p+1}(0,x)\varphi^2dx +C_\delta \iint_{\Omega_{t}} e^{-(t-s)}|\nabla v|^{\frac{2(p+2)}{m}}\varphi^2
+\kappa\iint_{\Omega_{t}} e^{-(t-s)}|\nabla v|^{p+2}\varphi^2 +C_\delta \kappa^{-N},
\end{align}
where $C_\delta$ is independent of $t$.

Next, recall that
\[
Z_{p+2}(t)=\sup_{x_0\in\R^N} \iint_{\Omega_{t}}e^{-(t-s)} u_+^{p+2}(s,x) \varphi^2(x-x_0)dxds.
\]
 Applying Lemma \ref{I3-lm} 
with $r=\frac{p+2}{m}$ and $r'=p+2>r$ (for $m>1$)  yields 
\begin{align}
\label{I3-eq1-2}
 C_\delta \iint_{\Omega_{t}}e^{-(t-s)} |\nabla v|^{\frac{2(p+2)}{m}}\varphi^2\leq\delta  Z_{p+2} (t)+C_\delta'\kappa^{-N}.
\end{align}
Applying Lemma \ref{I3-lm} 
with $r=\frac{p+2}{2}$ and $r'=p+2$ yields 
\begin{equation}
\label{I4-eq1-2}
\kappa\iint_{\Omega_{t}} e^{-(t-s)} |\nabla v|^{p+2}\varphi^ 2\le \kappa \delta  Z_{{p+2}}(t)+C_\delta\,\kappa^{1-N}.
\end{equation}
By \eqref{aux-lp-eq2}, \eqref{I3-eq1-2}, and \eqref{I4-eq1-2}, we have
\begin{align*}
&\int_{\R^N} u_+^{p+1}(t,x)\varphi^2(x)dx+c\iint_{\Omega_{t}}e^{-(t-s)} u_+^{p+m-2} |\nabla u_+|^2\varphi^2 dxds
+c\iint_{\Omega_{t}} e^{-(t-s)}u_+^{p+2}\varphi^2dxds\nonumber\\
&  \leq  \int_{\R^N} u_+^{p+1}(0,x)\varphi^2(x) dx+2\delta  Z_{p+2}(t)+C_\delta\,\kappa^{-N}.\nonumber
\end{align*}
Since $u_+(0,x)=u_0(x)+\eps$ ($\eps\in(0,1)$) is uniformly bounded, by Lemma \ref{psi-lm},
\[
\int_{\R^N} u_+^{p+1}(0,x)\varphi^2(x) dx\le C \kappa^{-N}.
\]
Then we fix $\delta:=c/4$ to get
\begin{align}
\label{estimate-case1-eq1}
&\int_{\R^N} u_+^{p+1}(t,x)\varphi^2(x)dx+c\iint_{\Omega_{t}}e^{-(t-s)} u_+^{p+2}\varphi^2dxds\nonumber\\
&\quad +c\iint_{\Omega_{t}} e^{-(t-s)}  u_+^{p+m-2} |\nabla u_+|^2\varphi^2 dxds
\le C  \kappa^{-N}+2^{-1}
cZ_{p+2}(t).
\end{align}

Now, by shifting the space variable, we can replace $\varphi(x)$ by $\varphi(x-x_0)$ in \eqref{estimate-case1-eq1}, and  obtain
\begin{align*}
&\int_{\R^N} u_+^{p+1}(t,x)\varphi^2(x-x_0)+c\iint_{\Omega_{t}} e^{-(t-s)}u_+^{p+2}\varphi^2(x-x_0)\\
&\quad +c\iint_{\Omega_{t}}e^{-(t-s)}  u_+^{p+m-2} |\nabla u_+|^2\varphi^2(x-x_0)\le C \kappa^{-N} +2^{-1}cZ_{p+2}(t)
\end{align*}
After taking supremum over $x_0$, this implies that
\begin{align}\label{nablaueps1}
&\sup_{x_0\in\R^N} \int_{\R^N} u_+^{p+1}(t,x)\varphi^2(x-x_0)+cZ_{p+2}(t)\nonumber\\
&\quad +c\sup_{x_0\in\R^N}\iint_{\Omega_{t}}e^{-(t-s)}  u_+^{p+m-2}|\nabla u_+|^2\varphi^2(x-x_0)
\leq C\kappa^{-N},
\end{align}
where $0<\kappa<\frac{c}{4}$, and the constants $c,C$ are independent of $\kappa$ and $t$.
Hence, recalling $u_+=u+\eps$, we get for all $p\geq m$,
$$
u(t,\cdot)\in L_{\rm loc}^{p+1}(\R^N), \quad e^{-T/2} (\eps+u)^{\frac{p+m-2}{2}}\nabla u\in L^2_{\rm loc}(\Omega_T),
$$
and 
$$
\sup_{t\in [0,T], x_0\in\R^N} \int_{\R^N} u^{p+2}(t,x)\varphi^2(x-x_0) \lesssim \kappa^{-N}.
$$
The bounds only depend only on $m,|\chi|,a,b,N,p$, $\|u_0\|_\infty$,  $\|v_0\|_{W^{1,\infty}}$  and the diameter of the local spatial domain (but independent of $T$). 

Finally, since $p\geq m$ can be arbitrary, picking $p=N+1\geq 2\geq m$ yields that $u(t,\cdot)$ is locally uniformly finite in $L^{N+2}(\R^N)$.
Then Lemma \ref{v-bound-lm}
 implies that
$$
\sup_{t\in [0,T],x\in\R^N}|\nabla v(t,x)|\lesssim \kappa^{-N}.
$$
This together with the fact that $\|v\|_\infty <\infty$  implies that $v \in L^\infty(\Omega_T)\cap  W^{1,\infty}(\Omega_T)$ with a bound independent of $T$.


\end{proof}

\subsection{Proof of Theorem \ref{apriori-boundedness-thm}(1)  for the case $m>2$}\lb{S.3.3}

In this subsection, we consider the case when $m>2$. Note that   the estimate \eqref{I2-eq1} for $\kappa I_2$  only holds when $1<m\le 2$, and when $m>2$,  we are not able to bound this term by the logistic term and the diffusion term. Therefore,
we will estimate both $I_1$ and $\kappa I_2$ differently, and a continuity argument will play an important role at the end.

We remark that the argument below actually works for the case of $m\in (1,2]$ as well. We treated the case $m\leq 2$ separately in the previous subsection as it allows a simpler proof.

\begin{proof}[Proof of Theorem \ref{apriori-boundedness-thm}(1)  for the case $m>2$]
\label{m-large-subsec}
Let $p\geq m$ and recall that $u_+=u+\eps$. 
We start with the following notations: 
For any given $x_0\in\R^N$ and $r>1$, define
\begin{equation*}
X_{ r,x_0}(t)=\int_{\R^N} u_+^r(t, x) \varphi^2(x-x_0)dx,\quad X_r(t)=\sup_{x_0\in\R^N} X_{r,x_0}(t),
\end{equation*}
\begin{equation}
\label{Y-eq}
\text{and}\quad Y_{r}(t):=\sup_{s\in[0,t]}X_{r}(s).
\end{equation}

The proof is divided into two cases.

\medskip

\noindent {\bf Case 1.} 
We first assume $N\geq 3$. We will make a better use of the term $I_1$ as follows:
\begin{align}\lb{I1-eq}
I_1\gtrsim \int_{\R^N} |\nabla (u_+^{\frac{p+m}{2}})|^2\varphi^2
\geq  c\int_{\R^N} |\nabla (u_+^{\frac{p+m}{2}}\varphi)|^2
-C\kappa^2 \int_{\R^N} u_+^{p+m}\varphi^2.
\end{align}
Applying Lemma \ref{I1-lm} with $r$ being replaced by $\frac{p+m}{2}$ (here we need $N\geq3$), we have
\begin{align}
\lb{new-poin'}
\kappa^2\int_{\R^N}u_+^{p+m}\varphi^2
&\leq  \delta\|\nabla(u_+^{\frac{p+m}{2}}\varphi)\|_{2}^2+C_{\delta}\kappa^{2+{\theta^*}}\left[\int_{\R^N}u_+^{\frac{p+m}{2}+1}\varphi^{q^*}dx\right]^{2/q^*}, 
\end{align}
where
\beq\lb{4.6}
q^*:=\frac{p+m+2}{p+m}\in (1, 2),
\quad{and}\quad
{\theta^*}:=\frac{N(p+m-2)}{p+m+2}>0.
\eeq
Similarly as done before, using the properties of $\varphi$ from Lemma \ref{psi-lm}, we have
\begin{align*}
\int_{\R^N}u_+^{\frac{p+m}{2}+1}(s,x)\varphi^{q^*}(x)dx&\lesssim \sum_{\kappa z\in\Z^N}\varphi^{q^*}(z)\int_{B_{N/\kappa}(z)}u_+^{\frac{p+m}{2}+1}(s,x)\varphi^2(x-z)dx\\
&\lesssim \sup_{x_0\in\R^N} \int_{\R^N}u_+^{\frac{p+m}{2}+1}(s,x)\varphi^2(x-x_0) dx.
\end{align*}
Since $
\frac{p+m}{2}+1\le p+1$  by \eqref{m-eq3}, $u_+^{\frac{p+m}{2}+1}\leq u_+^{p+1}+1$. By Lemma \ref{psi-lm} again,
\begin{align*}
\int_{\R^N}u_+^{\frac{p+m}{2}+1}(s,x)\varphi^{q^*}(x)dx&\lesssim \sup_{x_0\in\R^N} \int_{\R^N}u_+^{p+1}(s,x)\varphi^2(x-x_0) dx+\int_{\R^N}\varphi^2(x-x_0)dx\\
&\lesssim X_{p+1}(s)+\kappa^{-N}.
\end{align*}
It follows from \eqref{new-poin'} that
\beq
\lb{new-poin}
\kappa^2\int_{\R^N}u_+^{p+m}\varphi^2
\leq  \delta\|\nabla(u_+^{\frac{p+m}{2}}\varphi)\|_{2}^2+C_{\delta}\kappa^{2+{\theta^*}}X_{p+1}(s)^{2/q^*}+C\kappa^{2+{\theta^*}-2N/q^*}.
\eeq
After taking this $\delta>0$ to be small,
\eqref{I1-eq} and \eqref{new-poin} yield
\begin{equation}
\label{I1-eq2}
\frac12 I_1\geq c\|\nabla(u_+^{\frac{p+m}{2}}\varphi)\|_{2}^2-C\kappa^{2+{\theta^*}}X_{p+1}(s)^{2/q^*}-C\kappa^{2+{\theta^*}-2N/q^*}.
\end{equation}

Next, we bound $I_2$.
By Young's inequality  and   \eqref{new-poin}, we have for any $\delta\in (0,1)$, 
\begin{align}
\label{I2-eq2}
\kappa I_2
&\leq {\delta}  \int_{\R^N} u_+^{p+m-2}|\nabla u_+|^2 \varphi^2 + C\delta^{-1}\kappa^2 \int_{\R^N}u_+^{p+m}\varphi^2\nonumber\\
&\leq   {\delta}  \int_{\R^N}  u_+^{p+m-2} |\nabla u_+|^2 \varphi^2\,dx  +\delta\|\nabla(u_+^{\frac{p+m}{2}}\varphi)\|_{2}^2 \nonumber\\
&\quad +  C_{\delta}\kappa^{2+{\theta^*}}X_{p+1}(s)^{2/q^*}+C_{\delta}\kappa^{2+{\theta^*}-2N/q^*}.
\end{align}

By   \eqref{new-new-lp-eq1}, \eqref{I3-eq1}, \eqref{I4-eq1},  \eqref{I5-eq1},
 \eqref{I1-eq2}  and \eqref{I2-eq2},
we have for any $0<\delta <1$ sufficiently small depending only on $C$,  any $t\in [0,T]$, and  any  $0<\kappa < \delta$, there holds
\begin{align}
\label{aux-lp-eq3}
\frac{d}{dt}\int_{\R^N} u_+^{p+1}\varphi^2  dx
&\le -c \int_{\R^N}  u_+^{p+m-2} |\nabla u_+|^2\varphi^2 -c\int_{\R^N} u_+^{p+2}\varphi^2 -\int_{\R^N} u_+^{p+1}\varphi^2 \nonumber\\
&\quad +  C_\delta\int_{\R^N} |\nabla v|^{\frac{2(p+2)}{m}}\varphi^2+  \kappa\int_{\R^N} |\nabla v|^{p+2}\varphi^2+C_\delta \kappa^{-N}\nonumber\\
&\quad +  C_{\delta}\kappa^{2+{\theta^*}}X_{p+1}(s)^{2/q^*}+C_{\delta}\kappa^{2+{\theta^*}-2N/q^*}.
\end{align}
This together with \eqref{I3-eq1-2} and \eqref{I4-eq1-2}  implies that
\begin{align*}
&\int_{\R^N}u_+^{p+1}(t,x)\varphi^2(x)dx+c\iint_{\Omega_{t}} e^{-(t-s)} u_+^{p+m-2} |\nabla u_+|^2\varphi^2dxds+c\iint_{\Omega_{t}}e^{-(t-s)}u_+^{p+2}\varphi^2dxds\nonumber\\
&\le \int_{\R^N} u_+^{p+1}(0,x)\varphi^2dx +\delta  Z_{p+2} (t)+C_\delta \kappa^{-N}+ C_{\delta}\kappa^{2+{\theta^*}}Y_{p+1}(t)^{2/q^*}+C_{\delta}\kappa^{2+{\theta^*}-2N/q^*},
\end{align*}
where the constants $c,C_\delta$ are independent of $\kappa$ and $t$.
Using that $u_+(0,\cdot)=u_0(x)+\eps$  is uniformly bounded, Lemma \ref{psi-lm}, and fixing $\delta>0$ to be sufficiently small, we get
\begin{align*}
&\int_{\R^N} u_+^{p+1}(t,x)\varphi^2(x)dx+\iint_{\Omega_t} e^{-(t-s)} u_+^{p+m-2} |\nabla u_+|^2\varphi^2 dxds
+\iint_{\Omega_t} e^{-(t-s)}u_+^{p+2}\varphi^2dxds\nonumber\\
& \lesssim \delta Z_{p+2}(t) +\kappa^{2+{\theta^*}}Y_{p+1}(t)^{2/q^*}+\kappa^{2+{\theta^*}-2N/q^*}+\kappa^{-N}.
\end{align*}
Hence, for $x_0=0$,  since $Y_{p+1}(\cdot)$ is non-decreasing
\begin{align*}
&X_{p+1,x_0}(t)+\iint_{\Omega_t}e^{-(t-s)}   u_+^{p+m-2}|\nabla u_+|^2\varphi^2(x-x_0)dxds+ Z_{p+2,x_0}(t)\\
&\lesssim  \delta Z_{p+2}(t)+ \kappa^{2+{\theta^*}}Y_{p+1}(t_0)^{2/q^*}+\kappa^{2+{\theta^*}-2N/q^*}+\kappa^{-N}
\end{align*}
for $t_0\in (0,T]$ and $t\in [0,t_0]$.
After shifting in space, the same estimate holds for general $x_0\in\R^N$.  Taking supremum in $x_0$ and $t\in [0,t_0]$ for any $t_0\in [0,T]$, we get
\begin{align*}
&Y_{p+1}(t_0)+Z_{p+2}(t_0)+ \sup_{t\in [0,t_0], x_0\in\R^N} \iint_{\Omega_{t}} e^{-(t-s)}  u_+^{p+m-2} |\nabla u_+|^2\varphi^2(x-x_0)dxds\\
&\lesssim   \kappa^{2+{\theta^*}}Y_{p+1}(t_0)^{2/q^*}+\kappa^{2+{\theta^*}-2N/q^*}+\kappa^{-N}.
\end{align*}

Recall \eqref{4.6} and then
$$
2/q^*=\frac{2(p+m)}{p+m+2}\quad\text{and}\quad {2+{\theta^*}}{-2N/q^*}={2+\frac{N(p+m-2)}{p+m+2}- \frac{2N(p+m)}{p+m+2}}={2-N}.
$$
Therefore, there exists $C_0>0$ independent of $t_0$ such that for all $\kappa$ sufficiently small, 
\begin{align}
\label{estimate-case2-eq3}
&Y_{p+1}({t_0})
+\sup_{x_0\in\R^N} \iint_{\Omega_{t_0}} e^{-(t_0-s)}  u_+^{p+m-2} |\nabla u_+|^2\varphi^2(x-x_0)dxds\nonumber\\
&\leq C_0\kappa^{2+\frac{N(p+m-2)}{p+m+2}} \left[Y_{p+1}(t_0)\right]^{\frac{2(p+m)}{p+m+2}}+C_0\kappa^{-N}.
\end{align}

By further taking $C_0$ to be large enough if necessary, we can assume $Y_{p+1}(0)=X_{p+1}(0)\leq C_0$.
We claim that, if  $0<\kappa< 1$ is sufficiently small, 
\begin{equation}
\label{new-claim-eq}
Y_{p+1}(T)\le 2C_0 \kappa^{-N}.
\end{equation}
In fact,
 assume for contradiction that  there is $t_0\in [0,T]$ such that
$$
Y_{p+1}(t_0)=2C_0  \kappa^{-N}.
$$
By \eqref{estimate-case2-eq3},
\begin{align*}
2C_0\kappa^{-N}&\le C_0 \kappa^{2+\frac{N(p+m-2)}{p+m+2}} (2 C_0\kappa^{-N}) ^{\frac{2(p+m)}{p+m+2}}+ C_0\kappa^{-N}\\
&=C_0(2C_0)^{\frac{2(p+m)}{p+m+2}}\kappa^{2-N}+C_0\kappa^{-N}.
\end{align*}
This implies that
$$
1\le (2C_0)^{\frac{2(p+m)}{p+m+2}}\kappa^{2},
$$
which is impossible if $\kappa$ was chosen to be $\frac12 (2C_0)^{-\frac{p+m}{p+m+2}}$. Therefore, the claim \eqref{new-claim-eq}  holds.

By \eqref{estimate-case2-eq3}, we also
have
\begin{equation}\label{nablaueps2}
    \sup_{x_0\in\R^N} \iint_{\Omega_{T}}  e^{-(T-s)}  u_+^{p+m-2} |\nabla u_+|^2\varphi^2(x-x_0)dxds\le C_0(2C_0)^{\frac{2(p+m)}{p+m+2}}\kappa^{2-N}+C_0 \kappa^{-N}.
\end{equation}

The results yield that for any $p\geq m$, $u(t,\cdot)$ is locally uniformly finite in $L^{p+1}(\R^N)$ space for each $t\in [0,T]$, and  
$ e^{-T/2} (u+\eps)^{\frac{p+m-2}{2}}\nabla u$ is locally uniformly bounded in $L^{2}(\Omega_T)$.
Furthermore, after picking $p=\max\{m,N\}$, we obtain locally uniform boundedness of $u(t,\cdot)$ in $L^{N+1}(\bbR^N)$. By Lemma \ref{v-bound-lm}, $\nabla v$ is  uniformly finite in $L^{\infty}(\Omega_T)$.
Moreover,     the bounds of $u(t,\cdot)\in L_{\rm loc}^{p+1}(\R^N)$
and $v\in L^\infty(\Omega_T)\cap W^{1,\infty}(\Omega_T)$ are independent of $T$. This 
proves  Theorem \ref{apriori-boundedness-thm}(1)  for the case when $m>2$ and $N\ge 3$.

\medskip

\noindent{\bf Case 2.} $N=1,2$. 
 For any $t\ge 0$ and $x\in\R^N$, let $\tilde x=(x,x_{N+1}, x_{N+2})\in\R^{N+2}$
and 
\[
\tilde u(t,\tilde x)=u(t,x),\quad \tilde v(t,\tilde x)=v(t,x).
\]
Then $(\tilde u(t,\tilde x),\tilde v(t,\tilde x))$ is a solution of \eqref{main-perturbed-eq}
with $N$ being replaced by $N+2$ with extended initial data $(\tilde u_0,\tilde v_0)$.
The conclusion for the case $N=1,2$  thus follows from the case $N\geq 3$.
\end{proof}


\begin{rk}\lb{R.3.1}
Suppose that $\nabla v$ is a given $L^\infty$ vector field, and $u$ solves 
\[
u_t = m\nabla\cdot \big((\eps+u)^{m-1}\nabla u\big) - \chi \nabla \cdot (u \nabla v) + u(a - b u)\qquad\text{in $[0,T]\times\R^N$,}
\]
with bounded initial data. Then for any $p\geq m$, $u$ is locally uniformly finite in $L^p$ with a bound independent of $\eps$ and $T$, and $e^{-T/2} (\eps+u)^{\frac{p+m-2}{2}}\nabla u$ is locally uniformly bounded in $L^2(\Omega_T)$. 
Moreover, the bounds depend only on  
 $m,|\chi|,a,b,N,p$, $\|u_0\|_\infty$,  $\|\nabla v\|_{L^{\infty}}$  and the diameter of the local spatial domain.

The proof is the almost identical to the one of Theorem \ref{apriori-boundedness-thm}(1) and it is actually simpler. The only difference is that, in this case, we can replace \eqref{gradient-v-eq} by
\[
\iint_{\Omega_t} |\nabla v|^{2r}\varphi^2\leq C_r\,\kappa^{-N}
\]
with $C_r$ depending on $r$  and $\|\nabla v\|_{L^\infty}$.
\end{rk}

\section{$L^\infty$ a priori estimate for the perturbed problem and proof of  Theorem \ref{apriori-boundedness-thm}(2)}
In this section, we provide $L^\infty$ a priori estimate (independent of $\eps\in(0,1)$) and $T$  for solutions $(u,v)$ of \eqref{main-perturbed-eq} and prove Theorem \ref{apriori-boundedness-thm}(2). 
It follows from the last part of the proof of Theorem \ref{apriori-boundedness-thm} (1) that it suffices to consider the case where $N\geq 3$.
By Theorem \ref{apriori-boundedness-thm}(1), $u(t,\cdot)$ is locally uniformly finite in $L^p$ for all times $t$, though with a bound depending on $p$. 
We will apply Moser’s iteration method to upgrade the estimate to an $L^\infty$ bound for all times.

In the following, we assume $N\geq 3$, and  fix $T>0$ and 
\[
p_0 :=  \max\{N+1,m+1\}.
\]
By Theorem \ref{apriori-boundedness-thm}(1), there exists $K_0$ depending on $p_0$ 
(independent of $T$ and $\eps$) such that 
\begin{equation}
\label{u-p0-eq1}
 \sup_{t\in [0,T],x_0\in\R^N }  \|u(t,\cdot) \|_{L^{p_0}(B_1(x_0))} \le K_0.
\end{equation}
Since $p_0> N$, by Lemma \ref{v-bound-lm}, there exists $K_1$ depending only on  $p_0$, $N$, $\|v_0\|_{W^{1,\infty}}$  and $K_0$ such that
 \begin{equation}
\label{v-infinity-eq1}
    \|\nabla v(t, \cdot)\|_{\infty} \le K_1\quad \forall\, 0\le t\le T.
\end{equation}

Let $\varphi$ be from Lemma \ref{psi-lm} with a fixed parameter $\kappa\in (0,1)$.
We  multiply first equation by $u^{p}\varphi^2 $  and integrate it over  $\R^N$ with $p\geq m$. Applying Young's inequality, we get
\begin{align}
\label{lp-eq1}
&\frac{1}{p+1}\frac{d}{dt}\int_{\R^N} u^{p+1}\varphi^2  dx + \frac{1}{p+1}\int_{\R^N} u^{p+1}\varphi^2  dx\nonumber\\
&= \int_{\R^N} u^{p}\varphi^2 \nabla\cdot[m(u+\eps)^{m-1}\nabla u) - \chi u\nabla v] + (a+ \frac{1}{p+1})u^{p+1}\varphi^2  - bu^{p+2}\varphi^2 \nonumber\\
&= -m p \int_{\R^N} u^{p-1}\varphi^2  (u+\eps)^{m-1}|\nabla u|^2  \underbrace{ - m\int_{\R^N}u^{p}(u+\eps)^{m-1}\nabla u\cdot\nabla\varphi^2 }_{J_1}  \nonumber\\
&\quad +\underbrace{ \chi p\int_{\R^N} u^{p}\varphi^2 \nabla v\cdot\nabla u}_{J_2}+\underbrace{ \chi\int_{\R^N} u^{p+1} \nabla v\cdot\nabla \varphi^2  + \int_{\R^N} (a+ \frac{1}{p+1})u^{p+1}\varphi^2 - bu^{p+2}\varphi^2}_{J_3}.
\end{align}
Unlike in the previous Section \ref{S.3}, throughout this section, the general constants $c,C,C'\geq 1$ only depend on $m,\chi,a,b$, $N,K_0,K_1$,    $\|v_0\|_{W^{1,\infty}}$ and $\|u_0\|_{L^{\infty}}$, and they are likely to be different from one line to another. Let us emphasize that these general constants can depend on $p_0$ in this section, but they are absolutely independent of  $\eps, \kappa$, $T$, and $p$!

\subsection{Estimates for $J_1$, $J_2$ and $J_3$}

In this subsection,   we provide fine estimates for $J_1$, $J_2$ and $J_3$.

First, using that $|\nabla\varphi|\lesssim\kappa\varphi$, direct computation yields
\begin{align*}
J_1&\le  \frac{m p}{4}\int_{\R^N}  u^{p-1}\varphi^2  (u+\eps)^{m-1}|\nabla u|^2  + \frac{4 \k^2m}{p}\int_{\R^N} u^{p+1}(u+\eps)^{m-1}\varphi^2\nonumber\\
&\le \frac{m p}{4}\int_{\R^N}  u^{p-1}\varphi^2  (u+\eps)^{m-1}|\nabla u|^2  +  \frac{C \k^2}{p}\int_{\R^N} (u^{p+m}+1)\varphi^2,
\end{align*}
where in the second inequality we used that
\[
u^{p+1}(u+\eps)^{m-1}\leq 2^{m-1}u^{p+m}+2^{m-1}u^{p+1}\leq 2^m u^{p+m}+2^{m-1},
\]
and $C$ is independent of $\eps, \kappa$ and $p$. Then using \eqref{phi1} and \eqref{phi3} from Lemma \ref{psi-lm} yields
\begin{align}
\label{J1-eq}
J_1
&\le \frac{m p}{4}\int_{\R^N}  u^{p-1}\varphi^2  (u+\eps)^{m-1}|\nabla u|^2dx  +  \frac{C\k^2}{p}\int_{\R^N} u^{p+m}\varphi^2dx+   \frac{C \k^{2-N}}{p}\nonumber\\
&\le \frac{m p}{4}\int_{\R^N}  u^{p+m-2}|\nabla u|^2\varphi^2 dx  +  \frac{C\k^2}{p}\int_{\R^N} u^{p+m}\varphi^2dx+   \frac{C \k^{2-N}}{p}.
\end{align}

For $J_2$, since $|\nabla v|\leq C$ and $p\geq m-2$, we have
\begin{align}
\label{J2-eq}
J_2&\le  \frac{m p}{4}\int_{\R^N} u^{p+m-2}|\nabla u|^2\varphi^2  + {C p  }\int_{\R^N} u^{p+2-m} \varphi^2\nonumber\\
&\le   \frac{m p}{4}\int_{\R^N} u^{p+m-2}|\nabla u|^2\varphi^2dx  + Cp\int_{\R^N} u^{p+m} \varphi^2dx+ Cp \kappa^{-N},
\end{align}
where we used \eqref{phi1} and \eqref{phi3} again. Since $|\nabla v|\leq C$, $ m>1$ and $\kappa<1$,
\begin{align}
\label{J3-eq}
J_3&\le  C \int_{\R^N}  u^{p+1}\varphi^2-bu^{p+2}\varphi^2dx\le  C \int_{\R^N}  u^{p+m}\varphi^2dx+ C\kappa^{-N}.
\end{align}
Let us comment that here we did not use the assumption that $b>0$.
By \eqref{J1-eq} -- \eqref{J3-eq}, we have
\begin{equation}
\label{J1-J4-eq1}
J_1+J_2+J_3\le \frac{mp}{2}\int_{\R^N}  u^{p+m-2} |\nabla u|^2\varphi^2dx +C_1p\int_{\R^N}u^{p+m}\varphi^2dx +Cp\kappa^{-N},
\end{equation}
where $C_1$ and $C$ are general constants independent of $p$ and $\kappa$. 

Next,  we provide fine estimates for $\int_{\R^N}u^{p+m}\varphi^2$. 
Let
$$
\theta:=\frac{N(p+m-2)}{2(p+m+2)}\quad \text{and}\quad q:=\frac{p+m+2}{p+m}. 
$$
It follows from Lemma \ref{I1-lm} with $r=\frac{p+m}{2}$ that there is a dimensional constant $C>0$ such that
\begin{align*}
 \int_{\R^N} u^{p+m}\varphi^2
&\le \delta\left\|\nabla(u^\frac{p+m}{2}\varphi)\right\|_{2}^2+C\delta^{-\theta}\left[\int_{\R^N}u^\frac{p+m+2}2\varphi^{q}dx\right]^{2/q}, 
\end{align*}
Using the properties of $\varphi$, direct computation yields
\[
\left\|\nabla(u^\frac{p+m}{2}\varphi)\right\|_{2}^2\leq Cp^2 \int_{\R^N} u^{p+m-2}|\nabla u|^2 \varphi^2+C\kappa^2\int_{R^N} u ^{{p+m}}\varphi^2.
\]
 We pick $\delta:=c/p^2$ and $\kappa:=c$ with $c$ sufficiently small depending on $C_1$ from \eqref{J1-J4-eq1} to get
\begin{align*}
C_1p \int_{\R^N} u^{p+m}\varphi^2
&\le \frac{mp}{4} \int_{\R^N} u^{p+m-2}|\nabla u|^2 \varphi^2+Cp^{1+2\theta}\left[\int_{\R^N}u^\frac{p+m+2}2\varphi^{q}dx\right]^{2/q}.
\end{align*}

Then, let us bound $\int_{\R^N}u^\frac{p+m+2}2\varphi^{q}dx$.
Using Lemma \ref{psi-lm} and that $q\in (1,2)$, there is a dimensional constant $C>0$ such that for any $t\in [0,T]$ and $x_0\in\R^N$, 
\begin{align*}
\int_{\R^N} u^{\frac{p+m+2}{2}}(t,x) \varphi^q(x)dx &\le \int_{\R^N} u^{\frac{p+m+2}{2}}(t,x) \varphi(x)dx\nonumber\\
&\leq   C\sum_{\kappa z\in \Z^N} \int_{B_{{N/\kappa}}(z)} u^{\frac{p+m+2}{2}}(t,x)\varphi(x-x_0) dx\nonumber\\
&\le C  \sum_{\kappa z\in \Z^N}\varphi(z-x_0) \int_{B_{{N}/{\kappa}}(z)} u^{\frac{p+m+2}{2}}(t,x)\varphi^2(x-z)dx\nonumber\\
&\le CY_{\frac{p+m+2}{2}}(t)\sum_{\kappa z\in\Z^N} \varphi(z-x_0)\le  C Y_{\frac{p+m+2}{2}}(t),
\end{align*}
where we used the notation \eqref{Y-eq}.
This implies that
\begin{align}
\label{J-eq}
C_1p\int_{\R^N} u^{p+m}(t,\cdot)\varphi^2(\cdot) &\le    \frac{m p}{4}\int_{\R^N}u^{p+m-2}|\nabla u|^2\varphi^2 dx +Cp^{1+2\theta} Y_{\frac{p+m+2}{2}}(t)^{2/q}.
\end{align}

It follows from \eqref{J1-J4-eq1} and \eqref{J-eq} that
\begin{equation}
\label{J1-J4-eq2}
J_1+J_2+J_3\le \frac{3mp}{4}\int_{\R^N}  u^{p+m-2}|\nabla u|^2\varphi^2 dx + Cp^{1+2\theta} Y_{\frac{p+m+2}{2}}(t)^{2/q}
+Cp\kappa^{-N}.
\end{equation}

\subsection{Proof of Theorem \ref{apriori-boundedness-thm}(2)}

The proof of Theorem \ref{apriori-boundedness-thm}(2) follows from the following proposition.

\begin{prop}\lb{P.4.1}
    Suppose that $\nabla v$ is a given $L^\infty$ vector field,
    and $u$ solves 
\[
u_t = m\nabla\cdot \big((\eps+u)^{m-1}\nabla u\big) - \chi \nabla \cdot (u \nabla v) + u(a - b u)\qquad\text{in $[0,T]\times\R^N$,}
\]
with bounded initial data, and $u$ satisfies \eqref{u-p0-eq1}. Then $u$ is uniformly bounded in $[0,T]\times\R^N$ with the bound depending on general constants, but  independent of $\eps$ and $T$.  Moreover, for any $M>\|u_0\|_\infty$, there exists $T>0$ such that $\|u\|_\infty\leq M$ in $[0,T]\times\R^N$.
\end{prop}
\begin{proof}

First of all, by \eqref{lp-eq1} and \eqref{J1-J4-eq2},  there holds
\begin{align*}
&\frac{1}{p+1}\frac{d}{dt}\int_{\R^N} u^{p+1}\varphi^2  dx + \frac{1}{p+1}\int_{\R^N} u^{p+1}\varphi^2  dx\le  Cp^{1+2\theta} Y_{\frac{p+m+2}{2}}(t)^{2/q}
+Cp\kappa^{-N}
\end{align*}
where 
\[
\theta=\frac{N(p+m-2)}{2(p+m+2)}\in (cN, N)\quad\text{ for some }c\in(0,1)\text{ and }
\quad q=\frac{p+m+2}{p+m}.\]
Multiplying the above inequality on both sides by $e^t$, and integrating in time give 
\begin{align*}
&\int_{\R^N}u^{p+1}(t,x)\varphi^2(x)dx\\
&\le e^{-t}\int_{\R^N}u^{p+1}(0,x)\varphi^2(x)dx+\int_0^t e^{-(t-s)}(Cp^{2+2\theta} Y_{\frac{p+m+2}{2}}(t)^{2/q}
+Cp^2\kappa^{-N}) \, ds\\
&\le \int_{\R^N}u^{p+1}(0,x)\varphi^2(x)dx+Cp^{2+2\theta} Y_{\frac{p+m+2}{2}}(t)^{2/q}
+Cp^2\kappa^{-N},
\end{align*}
where we used that $Y_p(\cdot)$ is non-decreasing in time by its definition. This implies that for $t\in[0,T],$
\begin{align*}
\int_{\R^N}u^{p+1}(t,x)\varphi^2(x)dx&\le \int_{\R^N}u^{p+1}(0,x)\varphi^2(x)dx+Cp^{2+2\theta} Y_{\frac{p+m+2}{2}}(t)^{2/q}
+Cp^2\kappa^{-N}.
\end{align*}
Since $u(0,\cdot)$ is uniformly bounded, by shifting the space variable, we get for any $x_0\in\R^N$, 
\begin{align*}
\int_{\R^N}u^{p+1}(t,x)\varphi^2(x-x_0)dx&\le \max\left\{ Cp^{2+2\theta} Y_{\frac{p+m+2}{2}}(t)^{2/q},\, (C_\kappa)^p\right\}
\end{align*}
where $C_\kappa$ depends on general constants and $\k$. After taking supremum over $x_0$ and $t$, we get for $t\in [0,T]$,
\beq\lb{4.12}
Y_{p+1}(t)\le \max\left\{Cp^{2+2\theta} \left[Y_{\frac{p+m+2}{2}}(t)\right]^{\frac{2(p+m)}{p+m+2}},\, (C_\kappa)^p\right\}.
\eeq

Since $\theta\leq N$, if denoting
\begin{equation*}
W_{r}(t):= \max\left\{Y_r(t)^\frac{1}{r},\,C_\kappa\right\}
\end{equation*}
for some $C_\kappa$ depending only on general constants and $\kappa$, we get
\begin{align}
\label{lp-eq6}
W_{p+1}(t)&\leq \max\left\{C^\frac{1}{p+1}p^{\frac{2+2\theta}{p+1}} \left[Y_{\frac{p+m+2}{2}}(t)\right]^{\frac{2}{p+m+2}\cdot\frac{p+m}{p+1}},\, (C_\kappa)^\frac{p}{p+1}\right\}\nonumber\\
&\leq C^\frac{1}{p+1}p^{\frac{2+2N}{p+1}} \left[W_{\frac{p+m+2}{2}}(t)\right]^{\frac{p+m}{p+1}}.
\end{align}

Now, define $r_0:=p_0$ and, iteratively,
$$
r_n:=2r_{n-1}-m-1\quad\text{ for $n\geq1$}.
$$
It is easy to get
$$
r_n= 2^n(p_0-m-1)+m+1.
$$
Then \eqref{lp-eq6} becomes
\[
W_{r_n}(t)\leq C^\frac{1}{r_n}(r_n)^{\frac{2+2N}{r_n}} \left[W_{r_{n-1}}(t)\right]^{\frac{r_n+m-1}{r_n}},
\]
and then by iteration,
\begin{align}
\label{lp-eq13}
W_{r_n}(t)&\le C^{\frac{1}{r_n}+\frac{1}{r_{n-1}}(1+\frac{m-1}{r_n})}(r_n)^{\frac{2+2N}{r_n}}(r_{n-1})^{\frac{2+2N}{r_{n-1}}(1+\frac{m-1}{r_n})} \left[W_{r_{n-2}}(t)\right]^{(1+\frac{m-1}{r_{n-1}})(1+\frac{m-1}{r_n})} \nonumber\\
&\le\ldots\le \exp \Big({\sum_{j=1}^n\frac{\beta_{j+1}}{r_j}}\Big)\Big[ \prod_{j=1}^n (r_j)^{\frac{2+2N}{r_j}\beta_{j+1}}\Big] \left[W_{r_{0}}(t)\right]^{\beta_{1}}
\end{align}
where $\beta_{n+1}:=1$, and for $j=1,\ldots,n$,
\[
\beta_j:=(1+\frac{m-1}{r_n})\ldots (1+\frac{m-1}{r_j}).
\]

Since $p_0>m+1$, $C^{-1}2^n\leq r_n\leq C2^n$ for some $C>1$. Consequently, $\beta_j\in [1,C]$ for some general constant $C>1$ independent of $j$ and $n$, and so
\[
{\sum_{j=1}^n\frac{\beta_{j+1}}{r_j}}\leq C,
\]
\[
\ln \prod_{j=1}^n (r_j)^{\frac{2+2N}{r_j}\beta_{j+1}} =\sum_{j=1}^\infty\frac{(2+2N)\beta_{j+1}}{r_j}\ln r_j \leq  C\sum_{j=1}^\infty j2^{-j}<\infty.
\]
Hence $\prod_{j=1}^n (r_j)^{\frac{2+2N}{r_j}\beta_{j+1}}\leq C$ for some $C>0$ uniformly for all $n$.  Finally, since $W_{r_0}(t)=W_{p_0}(t)\leq C_\k$ by \eqref{u-p0-eq1}, we obtain from \eqref{lp-eq13} that there exists $C_\k$ depending only on $C$ and $\k$ such that
\[
W_{r_n}(t)\leq C_\k\quad\text{and then}\quad \sup_{x_0\in\R^N}\left[\int_{B_{N/\kappa}(x_0)}u^{r_n}(t,x)dx\right]^{1/r_n}\leq C_\k.
\]
After passing $n\to \infty$, this implies that there is $C$ independent of $\eps$  such that
$$
\|u(\cdot,\cdot)\|_{L^\infty([0,T]\times\R^N)}\le C.
$$

Finally, the last claim is clear from the proof. This is because when $T\leq 1$, the constants $C$ and $(C_\kappa)^p$ in \eqref{4.12} can be replaced by, respectively, $C'T$  and $C'_\kappa(p^2T+\|u_0\|_\infty^p)$ for some $C',C'_\kappa$ independent of $T\in (0,1]$. 
\end{proof}

\section{Global existence}

In this section, we first study the global well-posedness of weak/classical solutions of \eqref{main-perturbed-eq} with $\eps>0$  and prove Proposition \ref{main-perturbed-thm} in Subsection \ref{global-existence-proof}. Then we pass $\eps\to 0$ and prove Theorem \ref{main-thm} in Subsection \ref{s-main-thm}.  

\subsection{Proof of Proposition \ref{main-perturbed-thm}}
\label{global-existence-proof}

In this subsection, we prove Proposition \ref{main-perturbed-thm}. 
We do so by means of the following strategy. First, we establish the existence of a unique weak solution on an initial interval $[0,T]$ via the Banach fixed‑point theorem. Next, we extend this solution to a maximal interval $[0,T_{\max})$. Then we show that  the solution is a classical solution. Finally, 
 applying the previously derived a priori estimates, we demonstrate that $T_{\max} = \infty$.

Notice  that, in the case that $m=1$, the local existence of a unique classical solution can be achieved via semigroup approach. Indeed, one can show, via Banach fixed point theorem, the existence of a unique mild solution  satisfying
\begin{equation*}
\begin{cases}
u(t,\cdot)=e^{(\Delta -I)t}u_{0}-\chi \int_0^t e^{(\Delta-I)(t-s) }\nabla \cdot(u\nabla v)ds+\int_0^t e^{(\Delta -I)(t-s)} u(s,\cdot)\big(a+1-b u(s,\cdot)\big)ds\cr
v(t,\cdot)=e^{(\Delta-I)t}v_{0}+\int_0^t e^{(\Delta-I)(t-s)} v(s,\cdot)\big(1-u(s,\cdot)\big)ds,
\end{cases}
\end{equation*}
on some small interval $[0,T]$, see e.g., \cite{ hassan2024chemotaxis, winkler2010boundedness}. However, when $m > 1$, a semigroup representation is not available. In the literature on chemotaxis models with nonlinear diffusion, the local well-posedness of solutions to the perturbed problem is typically established using the theory of quasilinear parabolic equations (see, e.g., \cite{amann1993nonhomogeneous,ladyzhenskaya1968linear}). However, these arguments are often presented without full details. For completeness, we provide a detailed proof of the local well-posedness of solutions to \eqref{main-perturbed-eq}.

First of all, 
for given $T>0$ and $p>m+1$, let
\begin{equation}
\calX^p(T):=\{u\in C([0,T],L_{\rm loc}^p(\mathbb{R}^N))\},
\end{equation}
equipped with the norm
\[
\|u\|_{\calX^p(T)} := \frac{1}{|B_1|}\sup_{0\le t\le T,\,x_0\in\R^N}\|u(t,\cdot)\|_{L^p(B_1(x_0))}.
\]
Then $\calX^p(T)$ with the norm is a Banach space. Moreover, if $u\in L^\infty$, then $\|u\|_\infty\geq \|u\|_{\calX^p(T)}$.
For any $M> \|u_0\|_\infty$, let
\begin{equation*}
\mathcal{Z}^p(T,M)=\{u \in  \calX^p(T)\,|\, u(0,\cdot)=u_0,\,
 u\geq 0, \, \|u\|_{\calX^p(T)} \le M\}.
\end{equation*}
It is clear that $\mathcal{Z}^p(T,M)$ is a closed convex subset of $\calX^p(T)$.

The next lemma will be used to define the contraction mapping. 

\begin{lem}
\label{global-existence-lm2}
Assume that $u_0$ is uniformly {$C^{1+\alpha}$}, and $v_0$ is uniformly $C^{2+\alpha}$, $M>\|u_0\|_\infty$, $T\geq 1$, and $p>\max\{N,m+1\}$. For $i=1,2$ and any given $\tilde u_i\in \mathcal{Z}^p(T,M)$ such that $\tilde u_i$ is uniformly H\"{o}lder continuous in space and time, let $v_i$ be the classical solution to
\begin{equation*}
\begin{cases}
(v_i)_t=\Delta v_i-\tilde u_i (t,x) v_i,\quad & (t,x)\in \Omega_T,\cr
v_i(0,x)=v_{0}(x),\quad &x\in\R^N,
\end{cases}
\end{equation*}
and let $u_i$ be a classical solution to
\begin{equation*}
\begin{cases}
(u_i)_t=m\nabla \cdot \left((\eps + u_i)^{m-1} \nabla u_i\right)-\chi\nabla \cdot(u_i \nabla  v_i)+u_i(a-b  u_i),\,\, & (t,x)\in \Omega_T,\cr
u_i(0,x)= u_{0}(x),\,\, &x\in \R^N.
\end{cases}
\end{equation*}
Then there exists $T_1\in (0,1]$ depending only on $\eps, M,C$, $\|u_0\|_{C^{1+\alpha}}$ and $\|v_0\|_{C^{2+\alpha}}$ such that
\[
\|u_1-u_2\|_{\calX^p(T_1)}\leq \frac12\|\tilde u_1-\tilde u_2\|_{\calX^p(T_1)}.
\]
\end{lem}

\begin{proof}
Since $\tilde u_i$ is $C^\alpha$, by \cite[Theorem 8.1, Chapter V]{ladyzhenskaya1968linear}, $v_i(\cdot,\cdot)$ is uniformly $C^{1+\frac{\alpha}{2},2+\alpha}$ in $\Omega_T$ with norm depending on the H\"older norm of $\tilde u_i$ and$\|v_{0}\|_{C^{2,\alpha}}$, and then $u_i$ is locally  uniformly $C^{1+\frac{\alpha}{2},2+\alpha}$ in the interior of $\Omega_T$. By Lemma \ref{v-bound-lm}, 
since $\tilde u_i\in \mathcal{Z}^p(T,M)$ with $p>N$ is  uniformly bounded, there exists $K_1$ depending only $\|v_{0}\|_{W^{1,\infty}}$ 
 and $ M$ 
such that 
\beq\lb{4.1K}
|\nabla v_i(\cdot,\cdot)|\leq K_1\quad \text{ in }\Omega_T.
\eeq
By Proposition \ref{P.4.1}, there is $\tilde K$ 
  depending only on general constants and $K_1$
such that
\begin{equation}
\label{4.1Ktilde}
\|u_i(\cdot,\cdot)\|_\infty \le \tilde K \quad \text {in }\Omega_T.
\end{equation} 

Let $\varphi$ be from Lemma \ref{psi-lm} with parameter $\kappa=1$. It suffices to bound
\[
\int_{\R^N}|u_1(t,x)-u_2(t,x)|^p\varphi\, dx.
\]
Let 
$z = u_{1} - u_{2}$, then it satisfies
\begin{align}\label{zeqn1}
    z_{t} = &\nabla\cdot\left[m(u_{1}+\eps)^{m-1}\nabla z+ m\nabla u_{2}((u_{1}+\eps)^{m-1} - (u_{2}+\eps)^{m-1})\right] -\chi \nabla\cdot(z\nabla v_{1}) \nonumber\\
&+\chi\nabla\cdot( u_{2} \nabla(v_{2} -v_{1}))+ az - bz(u_{1} +u_{2})
\end{align}
in classical sense.
For some $p\geq m+1$ an even integer, multiplying \eqref{zeqn1} by $z^{p-1}\varphi$ and integrating it over $\R^N$ yield
\begin{align*}
    \frac{1}{p}\frac{d}{dt}\int_{\R^N} z^p\varphi
 & = -m(p-1)\int_{\R^N}z^{p-2} |\nabla z|^2 (u_{1} +\eps)^{m-1}\varphi  -m\int_{\R^N}z^{p-1} (u_{1} +\eps)^{m-1}( \nabla z\cdot\nabla\varphi)\nonumber\\
&\quad  -m(p-1)\int_{\R^N}z^{p-2} ((u_{1} +\eps)^{m-1} - (u_{2} + \eps)^{m-1})(\nabla z\cdot\nabla u_{2})\varphi \nonumber\\
&\quad  -m\int_{\R^N}z^{p-1} ((u_{1} +\eps)^{m-1} - (u_{2} + \eps)^{m-1}) (\nabla u_{2}\cdot \nabla \varphi)\nonumber\\
    &\quad  + \chi (p-1)\int_{\R^N}z^{p-1} (\nabla z\cdot\nabla v_1)\varphi + \chi \int_{\R^N}z^{p} (\nabla v_1\cdot \nabla \varphi)\nonumber\\
    &\quad- \chi (p-1)\int_{\R^N}u_{2}z^{p-2} (\nabla z \cdot\nabla(v_{2}-v_{1}))\varphi  - \chi\int_{\R^N}u_{2}z^{p-1}(\nabla(v_{2}-v_{1})\cdot \nabla \varphi ) \nonumber\\
    &\quad +\int_{\R^N} (az^p-bz^p(u_1+u_2))\varphi.
\end{align*}
It follows from \eqref{4.3} that $D^2v_i$ is bounded in $L^p_{\rm loc}$ for each $p\geq1$  (uniformly in space by shifting and locally uniformly in time). Also, recall \eqref{4.1K}, \eqref{4.1Ktilde}, and that $u_i(0,\cdot) =u_0(\cdot)$ is uniformly $C^{1+\alpha}$ and $v_i$ is uniformly bounded in $W^{1,\infty}$. By the second part of \cite[Theorem 3.1, Chapter V]{ladyzhenskaya1968linear},  we get $|\nabla u_i|\leq K_2$, for all $t\in [0,1]$ for some $K_2>0$ depending only on the general constants, $\eps, M$, $\|u_0\|_{C^{1+\alpha}}$ and $\|v_0\|_{C^{2+\alpha}}$.  
Hence, using these bounds and $|\nabla \varphi|\leq \varphi$ and $0\leq u_i\leq M$, we obtain for some $C>0$ depending only on the general constants, $\eps,M,p$, $\|u_0\|_{C^{1+\alpha}}$ and $\|v_0\|_{C^{2+\alpha}}$ such that for $t\in [0, 1],$

\begin{align}\label{fix-est1}
  \frac{1}{p}\frac{d}{dt}\int_{\R^N} z^p\varphi
 & \leq  -\frac{m(p-1)}{2}\int_{\R^N}z^{p-2} |\nabla z|^2 (u_{1} +\eps)^{m-1}\varphi   +C\int_{\R^N}z^{p}\varphi \nonumber\\
    &\quad +C\int_{\R^N}z^{p-2} \frac{|(u_{1} +\eps)^{m-1} - (u_{2} + \eps)^{m-1}|^2}{(u_{1} +\eps)^{m-1}}\varphi \nonumber\\
    &\quad + C\int_{\R^N}z^{p-1} |(u_{1} +\eps)^{m-1} - (u_{2} + \eps)^{m-1}|  \varphi \nonumber \\
    &\quad + C\int_{\R^N}z^{p-2} |\nabla z ||\nabla(v_{2}-v_{1})|\varphi  +C \int_{\R^N}|z^{p-1}||\nabla(v_{2}-v_{1})| \varphi  \nonumber\\
    &\quad    +C\int_{\R^N}z^{p-1} |\nabla z|\varphi +\int_{\R^N} (az^p-bz^p(u_1+u_2))\varphi,
\end{align}
where we also applied Young's inequality.

Direct computation yields that there is $C_\eps$ depending on $\eps,m,M$ such that
\[
|(u_{1} +\eps)^{m-1} - (u_{2} + \eps)^{m-1}| \le C_\eps|u_{1} -u_{2}| = C_\eps|z|.
\]
Since $p$ is even, $z^{p}$ and $z^{p-2}$ are non-negative.
By Young's inequality again, we have for any $\delta>0$,
\begin{align*}
z^{p-2} |\nabla z||\nabla(v_{2}-v_{1})| &\le \delta z^{p-2}|\nabla z|^2 \eps^{m-1}  + C_{\delta,\eps}|\nabla (v_{2} -v_{1})|^2z^{p-2}\\
    &\le  \delta z^{p-2}|\nabla z|^2(u_{1}+\eps)^{m-1}  + C_{\delta,\eps}|\nabla (v_{2} -v_{1})|^p+C_{\delta,\eps} z^{p},
\end{align*} 
and, similarly,
\begin{align*}
    |z|^{p-1} |\nabla(v_{2}-v_{1})| +|z|^{p-1}|\nabla z|&\le \delta z^{p-2}|\nabla z|^2(u_{1}+\eps)^{m-1}+C_{\delta,\eps} z^{p}+C|\nabla(v_2-v_1)|^p.
\end{align*} 
Fixing $\delta>0$ to be sufficiently small and plugging these into \eqref{fix-est1} yield
\begin{align*}
  &  \frac{1}{p}\frac{d}{dt}\int_{\R^N} z^p\varphi \le C_\eps \int_{\R^N} z^p\varphi + C_\eps \int_{\R^N}|\nabla(v_{2}-v_{1})|^p\varphi.
\end{align*}
This implies that for all  $t\in [0,1]$,
\begin{equation}\label{zpsi-est}
    \left\|(z\varphi_1)(t,\cdot)\right\|^p_{L^p} \le C_\eps  t\sup_{0<s<T}\left\|(\nabla(v_{2}-v_{1})\varphi_1)(s,\cdot)\right\|^p_{L^p}
\end{equation}
where $\varphi_1:=\varphi^{1/p}$ and the constant $C_\eps$ only depends on $C,\eps,M,p$, $\|u_0\|_{C^{1+\alpha}}$ and $\|v_0\|_{C^{2+\alpha}}$.

Now, note  that $v_{i}\varphi_1$ solves the equation 
$$
(v_{i}\varphi_1)_t = \Delta(v_{i}\varphi_1)  -[\tilde u_iv_{i}\varphi_1 + 2\nabla v_{i}\cdot\nabla\varphi_1+ v_{i}\Delta\varphi_1].
$$
Then, denoting $w:=(v_2-v_1)\varphi_1$ and recalling $T_p$ from section \ref{ss.2.1}, we get 
\begin{align*}
w(t,\cdot)&=T_p(t)w_0 + \int_0^t T_p({t-s}) \big[w-(\tilde u_2-\tilde u_1)v_{2}\varphi_1 -\tilde u_1w\\
&\quad - 2\nabla (v_{2}-v_1)\cdot\nabla\varphi_1- (v_{2}-v_1)\Delta\varphi_1 \big] ds .   
\end{align*}
Since $|\nabla\varphi|\leq \varphi$,  we have $|\nabla\varphi_1|\leq \varphi_1 $. In view of \eqref{Lp Estimates-2} and by Gr\"{o}nwall's inequality, we get for some $C$ and for all  $t\in[0,1]$,
\begin{align}\lb{5.12}
& \|w(t,\cdot)\|_{L^p}
\le C\int_0^t \|(\tilde u_{1} - \tilde u_2)(s,\cdot)\varphi_1\|_{L^p}ds +  \|w(s,\cdot)\|_{L^p} + C\int_0^t \|\nabla(v_{2}-v_{1})(s,\cdot)\varphi_1\|_{L^p}\nonumber\\
&\quad\le Ct \sup_{s\in[0,t]}\Big[ \|(\tilde u_{1} - \tilde u_2)(s,\cdot)\varphi_1\|_{L^p}ds +  \|w(s,\cdot)\|_{L^p} + \|\nabla(v_{2}-v_{1})(s,\cdot)\varphi_1\|_{L^p}\Big],
\end{align}
where we also used that $|v_2|\leq \|v_0\|_\infty, |\tilde u_1|\leq M$.
Similarly, by \eqref{Lp Estimates-3}, for   $t\in[0,1]$  we have 
\begin{align}\lb{5.13}
& \|\nabla w(t,\cdot)\|_{L^{p}}\nonumber
\le  C\int_0^t (t-s)^{-\frac{1}{2}}e^{-(t-s)}\|(\tilde u_{1} - \tilde u_2)(s,\cdot)\varphi_1\|_{L^p}\nonumber\\
&\quad + C\int_0^t (t-s)^{-\frac{1}{2}}e^{-(t-s)} \left[\|w(s,\cdot)\|_{L^p}+   \|\nabla(v_{2}-v_{1})(s,\cdot)\varphi_1\|_{L^p}\right]\nonumber\\
&\leq Ct^\frac12\sup_{s\in[0,t]}\Big[\|(\tilde u_{1} - \tilde u_2)(s,\cdot)\varphi_1\|_{L^p} + \|w(s,\cdot)\|_{L^p}+  \|\nabla(v_{2}-v_{1})(s,\cdot)\varphi_1\|_{L^p}\Big].
\end{align}
Then \eqref{5.12} and \eqref{5.13}  yield that for all  $t\in (0,1]$ being sufficiently small,
$$
\|\nabla(v_{2}-v_{1})(t,\cdot)\varphi_1\|_{L^p} \leq C(\| w(t,\cdot)\|_{L^{p}}+\|\nabla w(t,\cdot)\|_{L^{p}})\leq Ct^\frac12\sup_{0<s<t}\|(\tilde u_{1} - \tilde u_2)(s,\cdot)\varphi_1\|_{L^p}  .
$$

By \eqref{zpsi-est}, we have
\begin{align*}
    \|z(t,\cdot)\varphi_1\|^p_{L^p} 
       & \leq Ct^{1+\frac{p}{2}} \sup_{0<s<t}\|(\tilde u_{1}(s) - \tilde u_2(s))\varphi_1\|^p_{L^p}.
\end{align*}
Thus, if $T_1\le 1$ is sufficiently small depending on the general constants, $\eps,M,p$, $\|u_0\|_{C^{1+\alpha}}$ and $\|v_0\|_{C^{2+\alpha}}$, we conclude for all $t\in [0,T_1]$,
\begin{equation}
  \sup_{x_0\in\R^N}  \|(u_1-u_2)(t,\cdot)\|_{L^p(B_1(x_0))} \le \frac12\sup_{x_0\in\R^N,s\in[0,T_1]}\| (\tilde u_1 -\tilde u_2)(s,\cdot)\|_{L^p(B_1(x_0))} .
\end{equation}
\end{proof}

It follows from Lemma \ref{global-existence-lm2} that we can uniquely determine a function $u\in \mathcal{Z}^p(T_1,M)$ given $\tilde u\in\mathcal{Z}^p(T_1,M)$ and $\tilde u$ being H\"{o}lder continuous. We remove the H\"{o}lder continuity requirement in the next lemma and conclude  that the mapping
is a contraction $\mathcal{Z}^p(T_1,M)$.

\begin{lem}
\label{global-existence-lm1}
For any $M>\|u_0\|_\infty$ and $p>N$, there exists $T_1=T_1(M)\in (0,T]$ such that for any $\tilde u\in \mathcal{Z}^p(T,M)$, there exist $u\in \mathcal{Z}^p(T_1,M)$ and a bounded function $v$ such that they are weak solutions to
\begin{equation}
\label{fix}
u_t=m\nabla\cdot( (\eps +u)^{m-1}\nabla u)-\chi\nabla\cdot(u\nabla v)+u(a-bu),\quad u(x,0)=u_0
\end{equation}
and
\beq\lb{4.13}
v_t=\Delta v-\tilde u v,\quad v(x,0)=v_0,
\eeq
respectively, on $[0,T_1]$.
Moreover, there exists a mapping $\calL : \mathcal{Z}^p(T_1,M) \to  \mathcal{Z}^p(T_1,M)$ such that $\calL(\tilde u)=u$, where $u$ is a weak solution of \eqref{fix}
on $[0,T_1]$, and it  is a contraction on $\mathcal{Z}^p(T_1,M)$.

Lastly, if $\tilde u$  is uniformly H\"{o}lder continuous, then $u$ and $v$ are classical solutions to \eqref{fix} and \eqref{4.13}, respectively.

\end{lem}

\begin{proof}
First of all, we prove that for any $\tilde u\in \mathcal{Z}^p(T,M)$,   there exist $u\in \mathcal{Z}^p(T_1,M)$ and a bounded function $v$ such that they are weak solutions to \eqref{fix} and \eqref{4.13}
on $[0,T_1]$.

To this end, we take $\tilde u_\delta\in \mathcal{Z}^p(T,M)$ such that $\tilde u_\delta$ is uniformly H\"{o}lder continuous in space and time, and $ \|\tilde u_\delta-\tilde u\|_{\calX^p(T)}\to 0 $ as $\delta\to 0$.
Then take $v_\delta$ and $u_\delta$ as $v_i$ and $u_i$ from Lemma \ref{global-existence-lm2}, respectively, with $\tilde u_\delta$ in place of $\tilde u_i$.
It follows from \eqref{4.1K} that for some $K_1$ depending only on $\|v_{0}\|_{C^1}$, $\|\tilde u_\delta\|_{\calX^p(T)}$ and $M$ such that 
\[
|\nabla v_\delta(\cdot,\cdot)|\leq K_1\quad \text{ in }\Omega_T.
\]
Since $M>\|u_0\|_\infty$, by  Remark \ref{R.3.1} and Proposition  \ref{P.4.1}, there exist $T_1$ depending on $M$ but independent of $\delta$ and $\eps$ such that
\begin{equation}\label{udelta-bound}
\|u_{\delta}\|_{L^\infty(\Omega_{T_1})}  \le M,
\end{equation}
and,  for some $C$ independent of $\delta,\eps\in (0,1)$, 
\[
\sup_{x_0\in\R^N}\iint_{[0,T_1]\times B_1(x_0)}(\eps+u_\delta)^{2m-2}|\nabla u_\delta|^2dxdt\leq C.
\]

Next, passing $\delta\to 0$ along a subsequence, 
we can find $v\in L^\infty(0,T_1;W^{1,\infty}(\R^N))$ and  $u\in L^\infty(\Omega_{T_1})\cap L^2((0,T_1);W^{1,2}_{\rm loc}(\R^N))$ such that
\beq\lb{999}
\begin{aligned}
&v_\delta \rightharpoonup v, \quad \nabla v_\delta\rightharpoonup \nabla v,
\quad u_\delta \rightharpoonup u\quad\text{ in }L^p_{\rm loc}(\Omega_{T_1})\text{ for all }p\geq 1;
\\
&\qquad(\eps+u_\delta)^{m-1}\nabla u_\delta \rightharpoonup (\eps+u)^{m-1}\nabla u\quad\text{ in }L^2_{\rm loc}(\Omega_{T_1}).    
\end{aligned}
\eeq 
It is clear that $v$ is a weak solution of \eqref{4.13}  on $[0,T_1]$.
We claim   that $u\in \mathcal{Z}^p(T_1,M)$ and $u$ is a weak solution of \eqref{fix}. In fact, 
by  \eqref{udelta-bound} 
\begin{equation}
\label{new-u-bound-eq1}
\|u\|_{L^\infty(\Omega_{T_1})}\le M,
\end{equation}
and
 there exists a uniformly bounded vector field $g$ in $\Omega_{T_1}$ such that
\begin{equation}
u_\delta \nabla v_\delta \rightharpoonup g\quad\text{ in }L^p_{\rm loc}(\Omega_{T_1})\text{ for all }p\geq 1.
\end{equation}
We have from equation \eqref{4.1K} that $\nabla v_\delta $ is uniformly finite independent of $\delta$ (and also $\eps$), Lemma \ref{L.holder} below yields that $u_\delta$ is uniformly H\"{o}lder continuous in $[t_1,t_2]\times\R^N$ with $0<t_1<t_2\leq T_1$, and the H\"{o}lder norm is independent of $\delta,\eps\in (0,1)$. Therefore, (for fixed $\eps$) after passing $\delta\to0$ along a subsequence, we actually obtain that $u_\delta\to u$ pointwise locally uniformly in $(0,T_1]\times\R^N$. 
This and the weak convergence of $\nabla v_\delta\to \nabla v$ yield that 
\begin{equation}
\label{g-eq}
g=u\nabla v.
\end{equation}
By the $u_\delta $-equation (also see the computations of \eqref{lp-eq1}), $u_\delta(t,\cdot)$ is continuous in $t$ in the space of $L_{\rm loc}^p(\R^N)$ (uniformly in $\delta$). By the pointwise convergence of $u_\delta\to u$,  the boundedness of $u_\delta$ (see \eqref{udelta-bound}),  and  the Dominated Convergence Theorem,    we have that  for each $t\in [0,T_1]$, $u_\delta(t,\cdot)\to u(t,\cdot)$ in $L_{\rm loc}^p(\R^N)$,
and
\begin{equation}
\label{new-u-bound-eq4}
u\in \calX^p(T_1).
\end{equation}
By  \eqref{999}-\eqref{new-u-bound-eq4},  $u\in \mathcal{Z}^p(T_1,M)$ and $u$ is a weak solution of \eqref{fix} on $[0,T_1]$.  Without loss of generality, we may assume that $T_1$ satisfies the conclusion in  Lemma \ref{global-existence-lm2}.

\smallskip

In the following, we show that   there exists a mapping $\calL : \mathcal{Z}^p(T_1,M) \to  \mathcal{Z}^p(T_1,M)$ such that $\calL(\tilde u)=u$, where $u$ is a weak solution of \eqref{fix}
on $[0,T_1]$, and it  is a contraction on $\mathcal{Z}^p(T_1,M)$. 
Suppose that along another sequence of $\delta'\to 0$ (or we have a different approximations of $\tilde u_{\delta'}\to \tilde u$), for some $u'\in \mathcal{Z}^p(T_1,M)$ we have $u_{\delta'} \to u'$ in $L^p_{\rm loc}(\Omega_{T_1})$ and $u_{\delta'}(t,\cdot)\to u'(t,\cdot)$ in $L_{\rm loc}^p(\R^N)$ for each $t\in [0,T_1]$. It follows from Lemma \ref{global-existence-lm2} that
\[
\|u_{\delta'}-u_{\delta}\|_{\calX^p(T_1)}\leq \frac12\|\tilde u_{\delta'}-\tilde u_{\delta}\|_{\calX^p(T_1)}.
\]
Passing $\delta$ and $\delta'$ to $0$ yields that $u=u'$.
Hence the weak solution $u$ to \eqref{fix} on $[0,T_1]$ obtained via the above approximation process is unique.  This allows us to define  the mapping 
$\mathcal{L}: \mathcal{Z}^p(T_1,M)\to \mathcal{Z}^p(T_1,M_1)$,  where
$\calL(\tilde u)=u$ and $u$ is the weak solution of \eqref{fix} from the above approximation process.  Moreover,   Lemma \ref{global-existence-lm2} shows that this mapping $\mathcal{L}$  is a contraction on $\mathcal{Z}^p(T_1,M)$.

\smallskip

Finally,  if $\tilde u$  is uniformly H\"{o}lder continuous,  we can take $\tilde u_\delta=\tilde u$.
Hence, it follows from the previous proof that  $v=v_\delta$ is uniformly $C^{1+\frac{\alpha}{2},2+\alpha}$ in $\Omega_T$. The classical parabolic regularity results then yield that $u=u_\delta$ is locally  uniformly $C^{1+\frac{\alpha}{2},2+\alpha}$ in the interior of $\Omega_T$. So, they are classical solutions.
 \end{proof}

\medskip

The following lemma concerns both interior and bottom boundary H\"{o}lder regularity for nonlinear diffusion advection equation with logistic source 
\begin{equation}\lb{linearu}
u_t=m\nabla\cdot( (\eps +u)^{m-1}\nabla u)-\nabla\cdot(u V)+u(a-bu).
\end{equation}
The H\"{o}lder norm  depends only on $\|V\|_\infty$ and other parameters in the equation, and is independent of both higher regularity of $V$ and $\eps$. In the special case $a=b=\eps =0$, the argument for the interior estimate appears in \cite{kim2018regularity}, and the full proof of the general statement is given in \cite{hassan2025regularity}.

\begin{lem}\lb{L.holder}
Given a bounded vector field $V(t,x)$, $\eps>0$ and $\tau>0$, let $u\geq 0$ be a classical solution to \eqref{linearu}
in the domain of $(t_0, t_0+2\tau)\times B_2(x_0)$ for some $t_0\geq 0$ and $x_0\in\bbR^N$. Suppose that $\|u\|_\infty,\, \|V\|_\infty\leq M$ for some $M>0$. Then there exists $C$ depending only on $m,N,a,b,\tau$ and $M$ such that
\[
\|u\|_{C^{\alpha}((t_0+\tau, t_0+2\tau)\times B_1(x_0))}\leq C.
\]
Moreover, if $u(t_0,\cdot)$ is H\"{o}lder continuous in $B_2(x_0)$, $u(\cdot,\cdot)$ is H\"{o}lder continuous in $(t_0, t_0+2\tau)\times B_1(x_0)$ with a bound depending only on the $m,N,M,a,b,\tau$ and the H\"{o}lder norm of $u(t_0,\cdot)$. 
\end{lem}

\begin{rk}\lb{R.5.1}
In the proof of  \eqref{g-eq}, we invoke Lemma \ref{L.holder} to show that $u_\delta\to u$ as $\delta\to 0$   along a subsequence almost every where (in fact, pointwise) in $(0,T_1]\times\R^N$.   There is an alternative argument of proving this  fact without using Lemma \ref{L.holder}. We refer readers to \cite{jin2017boundedness, sugiyama2007time, tao2012global} for more details.
Indeed, by multiplying the $u_\delta$-equation by $u_\delta^{m}\zeta$ with $\zeta\in C_0^\infty(B_1)$, we get
\begin{align*}
\frac{1}{m+1}\int_{\R^N} (u_\delta^{m+1})_t\zeta  dx= -\int_{\R^N} \nabla(u_\delta^{m}\zeta)\cdot\left[m(u_\delta+\eps)^{m-1}\nabla u_\delta - \chi u_\delta\nabla v\right] + u_\delta^{m}\zeta(au_\delta - bu_\delta^2).
\end{align*}
By Young's inequality and Theorem \ref{apriori-boundedness-thm}, we obtain
\begin{align*}
\left|\int_{B_1} (u_\delta^{m+1})_t\zeta\right|&\leq C\left[1+\int_{B_1}(u_\delta+\eps)^{2m-2}|\nabla u_\delta(t,x)|^2dx+\|u_\delta\|_{2m}^{2m}+\|u_\delta\|_{m+2}^{m+2}\right]\|\zeta\|_{W^{1,\infty}_0(B_1)}\\
&\leq C\|\zeta\|_{W^{1,\infty}_0(B_1)}.    
\end{align*}
This implies that $\partial_t(u_\delta^{m+1})$ restricted to $(0,T_1)\times B_1$ is bounded under the norm of 
\[
L^1((0,T_1); (W^{1,\infty}_0(B_1))^*)=:L^1((0,T_1); X_1).
\]
Since $(\eps+u_\delta)^{m-1}\nabla u_\delta$ is locally uniformly bounded in $L^2(\Omega_{T})$ and $u_\delta$ is uniformly bounded, $u_\delta^{m+1}$   restricted to $(0,T_1)\times B_1$ is in the space of 
\[
L^2((0,T_1); W^{1,2}(B_1))=:L^2((0,T_1); X_0).
\]
Let $q\in (1,\frac{2N}{N-2})$ and $X:= L^q(B_1)$. Then
\[
X_0=W^{1,2}(B_1)\subseteq X\subseteq (L^\infty(B_1))^*\subseteq  (W^{1,\infty}_0(B_1))^*=X_1,
\]
and the embedding of $X\subseteq X_1$ is continuous.
By Rellich–Kondrachov embedding theorem, $X_0$ is compactly embedded in $X$. By Aubin-Lions compactness lemma, we have that, along a subsequence of $\delta\to 0$, {$u_\delta^{m+1}$} converges in the space of $L^2((0,T_1);L^q(B_1))$. This yields that $u_\delta\to u$  as $\delta\to 0$ along a subsequence a.e. in $[0,T_1]\times B_1$, which, after shifting, proves the claim.
\end{rk}

We now prove   Proposition \ref{main-perturbed-thm}.

\begin{proof}[Proof of Proposition \ref{main-perturbed-thm}]
First, 
fix $\eps>0$ and any $M>\|u_0\|_\infty$. Consider the mapping from Lemma \ref{global-existence-lm1}:
$$\tilde u\in \mathcal{Z}^p(T_1,M) \to \calL(\tilde u)=u(\cdot, \cdot)\in \mathcal{Z}^p(T_1,M).
$$
By Lemmas   \ref{global-existence-lm2} and \ref{global-existence-lm1}, if $T_1\le 1$ is picked to be sufficiently small, then the mapping is a contraction. 
By Banach fixed point theorem,
there is a unique $ u\in\mathcal{Z}^p(T_1,M)$ such that
$
\calL(u)= u.
$
By Lemmas \ref{global-existence-lm2} and \ref{global-existence-lm1} again, we obtain a 
weak solution $(u_\eps,v_\eps)$ to \eqref{main-perturbed-eq} in the time interval $[0,T_1]$.

Next,   we claim that $(u_\eps,v_\eps)$ is a classical solution  of    \eqref{main-perturbed-eq} on $(0,T_1]$. In fact,   recall that $u_{0}$ is uniformly H\"{o}lder continuous, and $v_{0}$ is uniformly $C^{2+\alpha}$. Let $v_{\eps}^\delta$ be uniformly smooth (depending on $\delta$), and for some $p>N$,
\[
\nabla v_\eps^{\delta} \to \nabla v_\eps \quad \text{ as $\delta\to 0$ in }L^p_{\rm loc}(\Omega_{T_1}).
\]
Then let $u_{\eps}^\delta$ solve \eqref{fix} with $v_{\eps}^{\delta}$ in place of $v$, in the classical sense. Since $\eps>0$, the classical parabolic theory yields that $u_\eps^\delta$ converges to $u_\eps$ as $\delta\to 0$ in $L^p_{\rm loc}(\Omega_{T_1})$. Since $\nabla v_\eps$ is uniformly bounded, $\nabla v_{\eps}^{\delta}$ can be chosen to be uniformly bounded independently of $\eps$ and $\delta$. Hence, Lemma \ref{L.holder} implies that $u_{\eps}^{\delta}$ is uniformly H\"{o}lder continuous in $\Omega_{T_1}$ with a bound independent of $\eps$ and $\delta$, which in turn implies that $u_\eps$ is also uniformly H\"{o}lder continuous independently of $\eps$. Thus, in view of Definition \ref{D.2}, the claim follows from the second part of Lemma \ref{global-existence-lm1}.


Now, we show that  \eqref{main-perturbed-eq} has a unique classical solution on $(0,T_1]$.
To  this end, 
 suppose that $(u_\eps',v_\eps')$ is another pair of classical solution of   \eqref{main-perturbed-eq} with initial condition $u_0,v_0$. Since $u_\eps$ and $u_\eps'$ are uniformly H\"{o}lder continuous in $[0,T_1]$, by the second part of Lemma \ref{global-existence-lm1}, 
\[
\calL(u_\eps)=u_\eps\quad\text{and}\quad \calL(u_\eps')=u_\eps'.
\]
Therefore, by Lemma \ref{global-existence-lm2}, for any $M>\|u_0\|_\infty$, there exists $0<\tilde T_1\le T_1$ such that $u_\eps,u_\eps'\in \mathcal{Z}^p(\tilde T_1,M)$ and
\[
\|u_\eps-u_\eps'\|_{\calX^p(\tilde T_1)}\leq \frac12\| u_\eps- u_\eps'\|_{\calX^p(\tilde T_1)}.
\]
This implies that $u_\eps=u_\eps'$ in $[0,\tilde T_1]\times\R^N$. Since both $u_\eps$ and $u_\eps'$ are uniformly bounded, by iteration, we actually get that $u_\eps=u_\eps'$ in $[0,T_1]\times\R^N$, which proves the uniqueness.

Finally,  by standard extension arguments,  we can extend the solution
$(u_\eps,v_\eps)$ of \eqref{main-perturbed-eq} on $[0,T_1]$   to a maximal interval $(0,T_{\max})$.  Moreover, if $T_{\max}<\infty$, then
\beq\lb{666}
\limsup_{t\to T_{\max}-}\|u_\eps(t,\cdot)\|_{C^{1+\alpha}}=\infty\quad {\rm or}\quad 
\limsup_{t\to T_{\max}-}\|v_\eps(t,\cdot)\|_{C^{2+\alpha}}=\infty.
\eeq
However,  Theorem \ref{apriori-boundedness-thm} implies
$$
\sup_{t\in[0,T_{\max})}\|u_\eps(t,\cdot)\|_{L^\infty}<\infty,\quad {\rm and}\quad \sup_{t\in [0,T_{\max})}\|\nabla v_\eps(t,\cdot)\|_{L^\infty}<\infty.
$$
By Lemma \ref{L.holder} again,
$u_\eps$ is uniformly H\"older continuous on $[0,T_{\max})$. 
In view of Lemma \ref{global-existence-lm1}, we have 
$$
\lim_{t\to T_{\max}-}\Big[\|u_\eps(t,\cdot)\|_{C^{1+\alpha}}+\|v_\eps(t,\cdot)\|_{C^{2+\alpha}}\Big]<\infty,
$$
which implies that \eqref{666} cannot occur. Consequently, $T_{\max}=\infty$, and this completes the proof of Proposition \ref{main-perturbed-thm}.

\end{proof}



\subsection{Proof of Theorem \ref{main-thm}}
\label{s-main-thm}


\begin{proof}
First,  
let $\eta\in C_c^\infty(\R^N)$ be non-negative with unit total mass. For any $\eps\in (0,1)$, define $\eta_\eps(x)=\eps^{-N}\eta(\eps^{-1}x)$, and let
$
u_{0,\eps}=u_0*\eta_\eps$ and 
$v_{0,\eps}=v_0*\eta_\eps$. Consequently, for any $p\geq 1$, we have
\[
\|v_{0,\eps}-v_0\|_{L^{p}_{\rm loc}},\, \|\nabla v_{0,\eps}-\nabla v_0\|_{L^{p}_{\rm loc}},\, \|u_{0,\eps}-u_0\|_{L^{p}_{\rm loc}}\to 0
\quad \text{ as }\eps\to 0.
\]
Moreover, $u_{0,\eps}\in L^\infty(\R^N)$, $v_{0,\eps}\in W^{1,\infty}$, and
$$
\|u_{0,\eps}\|_\infty\le \|u_0\|_\infty,\quad \|v_{0,\eps}\|_{W^{1,\infty}}\le \|v_0\|_{W^{1,\infty}}.
$$

Let $(u_\eps,v_\eps)$ from Proposition \ref{main-perturbed-thm} with $u_0$ and $v_0$ being replaced by
$u_{0,\eps}$ and $v_{0,\eps}$, respectively.
For any $T>0$, by Theorem \ref{apriori-boundedness-thm} and Proposition \ref{main-perturbed-thm}, there exits  $C$ independent of $\eps$ and $T>0$ such that
\beq\lb{678}
\|u_{\eps}\|_{L^\infty(\Omega_{T})}  \le C,\qquad
\sup_{x_0\in\R^N}\iint_{[0,T]\times B_1(x_0)}|\nabla u_\eps^m(t,x)|^2dxdt\leq C,
\eeq
and
\[
|\nabla v_\eps(\cdot,\cdot)|\leq C\quad \text{ in }\Omega_T.
\]
Also, we know that $v_\eps$ are uniformly bounded.
Therefore, after passing $\eps\to 0$ along a subsequence, 
we can find $v\in L^\infty(0,T;W^{1,\infty}(\R^N))$ and  $u\in L^\infty(\Omega_{T})$ such that
\beq\lb{convergence}
\begin{aligned}
&v_\eps\rightharpoonup v,\quad
\nabla v_\eps\rightharpoonup \nabla v,
\quad u_\eps \rightharpoonup u&&\quad\text{ in }L^p_{\rm loc}(\Omega_{T})\text{ for all }p\geq 1,
\\
&(\eps+u_\eps)^{m-1}\nabla u_\eps \rightharpoonup u^{m-1}\nabla u &&\quad\text{ in }L^2_{\rm loc}(\Omega_{T}).    
\end{aligned}
\eeq
Moreover, similarly as before, there is a uniformly bounded vector field $g$ in $\Omega_{T}$ such that
\[
u_\eps \nabla v_\eps \rightharpoonup g\quad\text{ in }L^p_{\rm loc}(\Omega_{T})\text{ for all }p\geq 1.
\]

Now, we show that $g=u\nabla v$. As $\nabla v_\eps $ is uniformly finite independent of  $\eps$, Lemma \ref{L.holder} yields that $u_\eps$ is uniformly H\"{o}lder continuous in $[T',T]\times\R^N$ with fixed $0<T'< T$, and the H\"{o}lder norm is independent of $\eps\in (0,1)$. This shows that $u_\eps\to u$ pointwise locally uniformly in $(0,T]\times\R^N$, and it follows that $g=u\nabla v$. We comment that this fact can also be obtained by the argument in Remark \ref{R.5.1} without invoking Lemma \ref{L.holder}.
Finally, due to \eqref{678} and \eqref{convergence}, $(u,v)$ is a global weak solution, and they stay uniformly bounded for all time.

\end{proof}

\appendix

\section{Appendix}
\subsection{Proof of Lemma \ref{maximal-regularity-lm}}

The proof closely follows that of \cite[Lemma 2.3]{hassan2024chemotaxis}, relying on the result from \cite{matthias1997heat}.
    
First, let
    $$
    \tilde g(t,x)=
    \begin{cases}
    g(t,x)\quad &{\rm for}\quad t\in (0,T),\,\, x\in\R^N\cr
    0\quad &{\rm for}\quad t>T,\,\, x\in\R^N.
    \end{cases}
    $$
     By \cite[Theorem 3.1]{matthias1997heat},  the initial value problem
    \begin{equation*}
    \begin{cases}
     \tilde v_t =\Delta \tilde v-\tilde v  + \tilde g,\quad x\in \R^N,\,\,  0<t<\infty \cr
    \tilde v(0,x) = 0
    \end{cases}
    \end{equation*}
    has a unique solution
     $\tilde v(\cdot,\cdot)\in W^{1,\gamma}((0,\infty),L^{\gamma}(\R^N))\cap L^{\gamma}((0,\infty), W^{2,\gamma}(\R^N))$.
Next, let
    $\tilde w(t,x):=e^{{ t}}\tilde v(t,x)$ which solves
\begin{equation}\label{pdelaplace-1}
    \begin{cases}
     \tilde w_t =\Delta  \tilde w + e^{{ t} }\tilde g,\quad x\in \R^N,\,\,  0<t<\infty \cr
    \tilde  w(0,x) = 0
    \end{cases}
    \end{equation}
We conclude that \eqref{pdelaplace-1}  has a unique solution in $W^{1,\gamma}((0,\infty),L^{\gamma}(\R^N))$ $\cap$ $L^{\gamma}((0,\infty), W^{2,\gamma}(\R^N))$. By the closed graph theorem, there is ${ C_\gamma}>0$ independent of $T$ such that
    \begin{equation}
    \label{est1}
      \int_0^T \int_{\R^N}e^{{  t}} \Big( |\tilde v(t,x)|^\gamma+|  \nabla \tilde v(t,x)|^\gamma+|\Delta\tilde  v(t,x)|^\gamma\Big)dxdt\le C_\gamma \int_0^T \int_{\R^N} e^{{ t}}| g(t,x)|^\gamma dxdt.
    \end{equation}

    Next,   let  $w:=v-\tilde {v}$, which  satisfies
    \begin{equation}
    \label{pdelaplace-2}
    \begin{cases}
     w_t =\Delta w -  w,\quad t
    \in (0,T),\, x\in \R^N\cr
    w(0,x) = v_0.
    \end{cases}
    \end{equation}
    By classical results of heat equations,
    \[
    w(t,x)= \int_{\R^{N}}e^{-t}G(t, x-y)v_0(y)dy=\int_{\R^N} e^{-t} G(t,y)v_0(x-y)dy,
    \]
    where $G$ is the heat kernel \eqref{heat-kernel}. Since
$    |\nabla G(t,x)|\leq \frac{|x|}{2t} G(t,x)$ and
 $v_0\in W^{1,\gamma}(\R^N)$, there holds
$$
\nabla w(t,x)=\int_{\R^N} e^{-t} G(t,x-y)\nabla v_0(y)dy
$$
and
$$
\Delta w(t,x)=\int_{\R^N} e^{-t} \nabla G(t,x-y)\cdot \nabla v_0(y)dy.
$$
    Then, by Young's convolution inequality,   there is $C_{\gamma,N}>0$ such that  for all $t>0$,
    $$
    \|w(t,\cdot)\|^{\gamma}_{L^\gamma(\R^N)} +\|\nabla w(t,\cdot)\|^{\gamma}_{L^\gamma(\R^N)}+ \|\Delta  w(t,\cdot)\|^{\gamma}_{L^\gamma(\R^N)}
   \leq C_{\gamma,N} e^{-t} 
    \|v_0(\cdot)\|^{\gamma}_{W^{1,\gamma}(\R^N)}.
    $$
    This implies that  
    \begin{align*}
    \int_{0}^T \int_{\R^N}e^{t}\Big( |w |^\gamma +|\nabla w|^\gamma+|\Delta w|^\gamma  \Big)dxdt \le C_{\gamma,N}T \|v_0(\cdot)\|^{\gamma}_{W^{1,\gamma}(\R^N)} .
    \end{align*}
    The above and \eqref{est1} yield the conclusion.

\subsection{Proof of Lemma \ref{v-bound-lm}}

Let $\varphi$ from Lemma \ref{psi-lm} with some $\kappa>0$ to be determined. First,  note that
\begin{align}
\label{v-phi-eq}
(v\varphi)_t=\Delta(v\varphi)-v\varphi +\left(v\varphi-2\nabla v\cdot \nabla\varphi - v\Delta \varphi - uv \varphi\right)
\end{align}
and $v_0\varphi\in W^{1,p}(\R^N)\cap W^{1,\infty}(\R^N)$ for any $p>1$.
Hence for any $p\le q\le\infty$, 
\begin{align}
\label{A-eq}
\| (\nabla  v (t,\cdot))\varphi\|_{L^{q}}
&\le \|v(t,\cdot)\nabla \varphi\|_{L^{q}}+\|\nabla(v\varphi)\|_{L^{q}}\nonumber\\
&\le \underbrace{ \|v(t,\cdot)\nabla \varphi\|_{L^{q}}}_{A_1(q)}+ \underbrace{ \|\nabla e^{(\Delta- I)t} (v_0\varphi)\|_{L^{q}}}_{A_2(q)}\nonumber\\
&\quad +
 \underbrace{\int_0^t \left\|\nabla e^{(\Delta- I)(t-s)}\Big( v\varphi-2\nabla v\cdot\nabla \psi-v\Delta\varphi-\varphi uv\Big)\right\|_{L^{q}}ds}_{A_3(q)}.
\end{align}

Next, we estimate $A_1(q)$, $A_2(q)$, and $A_3(q)$ for  $q=p$ or  $q=\infty$. Below, we write $C_\kappa$ as a constant that might depend on $C,p$ and $\kappa$, while $C$ is independent of $\kappa$.
Since $v$ is uniformly bounded,
\begin{equation}
\label{A1-eq}
A_1(p)\leq C_\kappa\quad {and}\quad A_1(\infty)\leq C\kappa.
\end{equation}
Next, it follows from \eqref{Lp Estimates-2}, \eqref{L-infty- Estimates-1} and $\|v_0\|_{W^{1,\infty}}\leq C$ that
\beq
\begin{aligned}
    \label{A2-eq}
A_2(p)\leq C\|\nabla (v_0\varphi)\|_{L^p}\leq C_\kappa \quad\text{and}\quad
A_2(\infty)\leq  C.
\end{aligned}
\eeq
Applying \eqref{Lp Estimates-3} with $q=p$  to $A_3(p)$ yields
\begin{align}
\label{A3-eq1}
A_3(p)&\leq C\int_0^t e^{-(t-s)}(t-s)^{-\frac{1}{2}}\| v\varphi-2\nabla v\cdot\nabla \varphi-v\Delta\varphi-\varphi uv  \|_{L^p}\nonumber\\
&\leq C_\kappa +C\int_0^ t e^{-(t-s)}(t-s)^{-\frac{1}{2}} \Big(\|u\varphi\|_{L^p}+\|\nabla v\cdot\nabla  \varphi\|_{L^p}\Big)\nonumber\\
&\leq C_\kappa +C\sup_{s\in [0,T]} \|u(s,\cdot)\varphi\|_{L^p}+C\kappa \int_0^ t e^{-(t-s)}(t-s)^{-\frac{1}{2}}\|(\nabla v(s,\cdot))\varphi\|_{L^p} ds.
\end{align}

Now, by \eqref{A-eq}, \eqref{A1-eq}, \eqref{A2-eq}, and \eqref{A3-eq1}, we have
\begin{align*}
\| (\nabla  v (t,\cdot))\varphi\|_{L^{p}}
&\le C_\kappa+C\sup_{s\in [0,T]} \|u(s,\cdot)\varphi\|_{L^p}+C\kappa \sup_{0\le s\le T} \|(\nabla v(s,\cdot))\varphi\|_{L^p}\quad \forall\, 0\le t\le T.
\end{align*}
By taking supremum in $t\in [0,T]$ and taking $\kappa$ to be sufficiently small depending on $C$, we get
\begin{equation}
\label{A-eq1}
\sup_{0\le t\le T}\|(\nabla v(t,\cdot))\varphi\|_{L^p}\le C_\kappa.
\end{equation}
In the rest of the proof, we fix one such $\kappa$ and we might drop it from the notations of $C_\kappa$.

Finally, applying \eqref{Lp Estimates-3} with $q=\infty$ to $A_3(p)$ yields that for $p>N$,
\begin{align*}
A_3(\infty)&\leq C\int_0^t e^{-(t-s)}(t-s)^{-\frac{1}{2}-\frac{N}{2p}}\| v\varphi-2\nabla v\cdot\nabla \varphi-v\Delta\varphi-\varphi uv  \|_{L^p}\nonumber\\
&\leq C_\kappa +C\int_0^ t e^{-(t-s)}(t-s)^{-\frac{1}{2}-\frac{N}{2p}} \Big(\|u\varphi\|_{L^p}+\|\nabla v\cdot\nabla  \varphi\|_{L^p}\Big)\nonumber\\
&\leq C_\kappa +C\sup_{s\in [0,T]} \|u(s,\cdot)\varphi\|_{L^p}+C\kappa \int_0^ t e^{-(t-s)}(t-s)^{-\frac{1}{2}-\frac{N}{2p}}\|(\nabla v(s,\cdot))\varphi\|_{L^p} ds.
\end{align*}
Since $\int_{B(x_0,1)} u^p(t,x)dx\leq C_1$ uniformly for all $t\in [0,T]$ and $x_0\in\R^N$ by the assumption,
this and \eqref{A-eq1} imply that
\begin{equation}
\label{A3-eq2}
A_3(\infty)\le C\quad\text{ for some }C=C(C_1).
\end{equation}
By \eqref{A-eq}, \eqref{A1-eq}, \eqref{A2-eq}, and \eqref{A3-eq2}, there holds
$$
\| (\nabla  v (t,\cdot))\varphi(\cdot)\|_{L^{\infty}}\le C\quad \forall\, t\in [0,T].
$$
Replacing $\varphi(\cdot)$ by $\varphi(\cdot -x_0)$ yields
$
\| (\nabla  v (t,\cdot))\varphi(\cdot-x_0)\|_{L^{\infty}}\le C$ uniformly for all $t\geq 0$ and $x\in\bbR^N$.
This implies that
$$
\sup_{0\le t\le T} \|\nabla v(t,x)\|_{L^\infty}\le C
$$
for $C$ depending only on $p$, $N$,  $ C_1$ and $ \|v_0\|_{W^{1,\infty}}$.
The lemma is thus proved.

\bibliographystyle{siam}

\end{document}